\documentclass[a4paper,oneside,reqno]{amsart}
\usepackage[top=30mm,bottom=30mm,left=30mm,right=30mm]{geometry}

\usepackage[USenglish]{babel}
\usepackage{csquotes} 
\usepackage{comment}

\usepackage{mathtools,amssymb,amsthm}
\usepackage[arrow,matrix]{xy} 
\usepackage[shortlabels]{enumitem} 

\usepackage{graphicx}
\usepackage{xcolor}

\usepackage[
 bookmarksnumbered=true,
 colorlinks=true,
 allcolors=blue
]{hyperref}

\usepackage[
 backend=biber, 
 style=alphabetic,
 sorting=nyvt
]{biblatex}
\theoremstyle{plain}
 \newtheorem{theorem}{Theorem}[section]
 \newtheorem{proposition}[theorem]{Proposition}
 \newtheorem{lemma}[theorem]{Lemma}
 \newtheorem{corollary}[theorem]{Corollary}

 \newtheorem{Itheorem}{Theorem}
 
 \newtheorem{Icorollary}{Corollary}

\theoremstyle{definition}
 \newtheorem{definition}[theorem]{Definition}
 \newtheorem{example}[theorem]{Example}
 \newtheorem{remark}[theorem]{Remark}
 
 \newtheorem{Idefinition}{Definition}
 
  \newtheorem{Inotation}{Notation}

\numberwithin{equation}{section}

\title[\'{E}tale Fund. Groups of Smooth Arith. Surfaces and the Grothendieck Conj.]
      {\'{E}tale Fundamental Groups of Smooth Arithmetic Surfaces and the Grothendieck Conjecture}
\date{Version of \today}

\author[R.~Shimizu]{Ryoji Shimizu}
\address{Institute of Science Tokyo, 2-12-1 Ookayama, Meguro-ku, Tokyo 152-8550, Japan}
\email{shimizu.r.ap@m.titech.ac.jp}

\author[N.~Yamaguchi]{Naganori~Yamaguchi}
\address{Tokyo Denki University, 5 Senju-Asahi-cho, Adachi-ku, Tokyo 120-8551, Japan}
\email{n.yamaguchi@mail.dendai.ac.jp}

\thanks{This work was supported by JSPS KAKENHI Grant Numbers 25KJ0125 and 23KJ0881.}

\DeclareMathOperator{\Hom}{Hom}
\DeclareMathOperator{\Isom}{Isom}
\DeclareMathOperator{\Aut}{Aut}
\DeclareMathOperator{\Out}{Out}
\DeclareMathOperator{\Inn}{Inn}

\DeclareMathOperator{\Spec}{Spec}
\DeclareMathOperator{\Gal}{Gal}

\begin{document}

\begin{abstract}
    We study the structure of the \'etale fundamental groups of smooth curves over certain arithmetic schemes,
    and investigate the relative version of Grothendieck's anabelian conjecture in this setting.
    Consequently, every hyperbolic curve over the ring of \texorpdfstring{$S$}{S}-integers of a number field in which a rational prime is inverted is anabelian, i.e., its schematic structure is completely determined by its \'etale fundamental group.
    Moreover, we obtain a partial result toward the semi-absolute version of Grothendieck's anabelian conjecture in this context.
\end{abstract}

\maketitle
\tableofcontents

\section*{Introduction}\label{intro}
In Grothendieck's letter \cite{MR1483108} to G.~Faltings, the following conjecture, called the \textit{Grothendieck conjecture}, was proposed:

\medskip

\begin{quote}\textit{The geometry of ``anabelian'' schemes over fields that are finitely generated over $\mathbb{Q}$ is completely determined by their \'etale fundamental groups.}
\end{quote}

\medskip

Although A.~Grothendieck did not specify precisely which schemes should be called ``anabelian'', he conjectured that hyperbolic curves are anabelian.
The Grothendieck conjecture for hyperbolic curves was proved in the genus zero case by H.~Nakamura \cite[Theorem~6.1]{MR1072981}, in the affine case by A.~Tamagawa \cite[Theorem~6.3]{MR1478817}, and, in a general form stronger than originally conjectured, by S.~Mochizuki \cite[Theorem~16.5]{MR1720187}, who proved that:
\medskip
\begin{quote}\textit{For any rational prime $p$ and hyperbolic curves $X_{1}$, $X_{2}$ over a generalized sub-$p$-adic field $k$ (see~Definition~\ref{def:gen-sub-p-adic}), the natural map
    \begin{equation*}
        \Isom_{k}(X_{1},X_{2})
        \longrightarrow
        \Isom^{\Out}_{G_{k}}\bigl(\pi^{\mathrm{\acute{e}t}}_{1}(X_{1})^{(p)},\,\pi^{\mathrm{\acute{e}t}}_{1}(X_{2})^{(p)}\bigr)
    \end{equation*}
    is bijective.
    Here $\pi^{\mathrm{\acute{e}t}}_{1}(X_{i})^{(p)}$ denotes the maximal geometrically pro-$p$ quotient of the \'etale fundamental group $\pi^{\mathrm{\acute{e}t}}_{1}(X_{i})$ of $X_{i}$ (see~Notation~\ref{notation:etale}), $G_{k}$ denotes the absolute Galois group of $k$, and $\Isom_{\ast}^{\Out}(\ast,\ast)$ denotes the set of all outer isomorphisms (see Notation~\ref{nor:Isomset_prof}).}
\end{quote}
\medskip
Subsequent to S.~Mochizuki's proof of the Grothendieck conjecture for hyperbolic curves, Y.~Hoshi proved that hyperbolic polycurves of dimension $\leq 4$ are anabelian (see~\cite[Theorem~B]{MR3288808}).
Moreover, A.~Schmidt and J.~Stix \cite{MR3549624} proved that strongly hyperbolic Artin neighborhoods are anabelian and, using this, further showed that any smooth, geometrically connected variety over a field finitely generated over $\mathbb{Q}$ admits a basis of Zariski neighborhoods consisting of anabelian varieties (see~\cite[Corollary~1.7]{MR3549624}), as predicted by A.~Grothendieck in his letter \cite{MR1483108} to G.~Faltings.
Y.~Hoshi also proved a more general statement by a different method (see~\cite[Corollary 3.4]{MR4194185}).

In the present paper, we prove the Grothendieck conjecture for hyperbolic curves over certain arithmetic schemes.
Our first main theorem is as follows.

\begin{Itheorem}[Theorem~\ref{thm:relative_anabelian}]\label{Ithm:relative_anabelian}
    Let $\mathcal{S}$ be a scheme that satisfies the following condition:
    \begin{enumerate}[(A)]
        \item\label{Ithm:relative_anabelian_condition0}
              $\mathcal{S}$ is a connected, normal, locally Noetherian scheme whose function field is a generalized sub-$p$-adic field for some rational prime $p$ that is invertible on $\mathcal{S}$.
    \end{enumerate}
    Then, for any hyperbolic curves $\mathcal{X}_{1}$, $\mathcal{X}_{2}$ over $\mathcal{S}$, the natural map
    \begin{equation*}
        \Isom_{\mathcal{S}}(\mathcal{X}_1,\mathcal{X}_2)
        \longrightarrow
        \Isom^{\Out}_{\pi^{\mathrm{\acute{e}t}}_{1}(\mathcal{S})}\bigl(\pi_{1}^{\mathrm{\acute{e}t}}(\mathcal{X}_1)^{(p)},\pi_{1}^{\mathrm{\acute{e}t}}(\mathcal{X}_2)^{(p)}\bigr)
    \end{equation*}
    is bijective.
\end{Itheorem}
The condition \ref{Ithm:relative_anabelian_condition0} is satisfied, for instance, by the spectrum of every generalized sub-$p$-adic field (e.g., every number field), and by the spectrum of any ring of $S$-integers in which a rational prime is inverted (e.g., $\mathbb{Z}[\tfrac{1}{p}]$).
Moreover, if a scheme $\mathcal{S}$ satisfies \ref{Ithm:relative_anabelian_condition0}, then any connected, normal scheme of finite type over $\mathcal{S}$ also satisfies \ref{Ithm:relative_anabelian_condition0} (e.g., $\mathbb{P}^{n}_{\mathcal{S}}$ and $\mathbb{A}^{n}_{\mathcal{S}}$ for $n\geq 1$), see~Example~\ref{ex:conditiona}.
In particular, Theorem~\ref{Ithm:relative_anabelian} gives another class of anabelian schemes:

\begin{Idefinition}\label{def:arithmetic_surface}
    A scheme $X$ over a Dedekind scheme $\mathcal{S}$ is called an \textit{arithmetic surface} if $X$ is regular and the structural morphism $f\colon X\to \mathcal{S}$ is projective, flat, of relative dimension one, and has geometrically connected generic fiber (cf.~\cite[Definition~8.3.14]{MR1917232} for a slightly different definition).
    If, moreover, $f$ is smooth, then $X$ is called a \textit{smooth arithmetic surface}.
\end{Idefinition}

\begin{Icorollary}\label{Icor:arith-surf-GC}
    Let $K$ be a number field and let $S$ be a set of primes of $K$.
    If there exists a rational prime that is invertible in the ring $\mathcal{O}_{K,S}$ of $S$-integers, then every smooth arithmetic surface over $\mathcal{O}_{K,S}$ whose generic fiber has genus at least $2$ is anabelian.
\end{Icorollary}

The hypothesis ``there exists a rational prime that is invertible on $\mathcal{S}$" in Theorem~\ref{Ithm:relative_anabelian} is necessary.
Indeed, if no rational prime is invertible on $\mathcal{S}$, then we can construct counterexamples to Theorem~\ref{Ithm:relative_anabelian} by using techniques originally given by T.~Saito in \cite{MR1305400}.
We present these counterexamples in Subsection~\ref{subsection1.2}.

\medskip

In Theorem~\ref{Ithm:relative_anabelian}, we assume that every isomorphism is defined over $\mathcal{S}$ or over $\pi_{1}(\mathcal{S})$.
On the other hand, we now consider a \emph{semi-absolute} version in which isomorphisms are not required to be defined over the fixed base (see~Definition~\ref{def:GP}).
The Neukirch--Uchida theorem implies that the semi-absolute and relative versions of the Grothendieck conjecture over number fields are equivalent.
Similarly, from Theorem~\ref{Ithm:relative_anabelian} and the Neukirch--Uchida type result \cite[Theorem~3.4]{MR4626871} for the ring $\mathcal{O}_{K,S}$ of $S$-integers provided $S$ is sufficiently large, we obtain the following result for the semi-absolute version of the Grothendieck conjecture over rings of $S$-integers:

\begin{Itheorem}[Remark~\ref{rem:shimizu} and Theorem~\ref{thm:semi-absolute-main}]\label{Ithm:htrsjthtr}
    Let $i$ range over $\{1,2\}$.
    Let $K_i$ be a number field and let $S_i$ be a set of primes of $K_i$.
    Let $\mathbb{L}(S_{i})$ be the set of all rational primes that are invertible in $\mathcal{O}_{K_{i},S_{i}}$ and let $\mathbb{L}_{i}\subset \mathbb{L}(S_{i})$ be a nonempty subset.
    Let $\mathcal{X}_{i}$ be a hyperbolic curve over $\mathcal{O}_{K_{i},S_{i}}$.
    If the upper Dirichlet densities of $\mathbb{L}(S_1)$ and $\mathbb{L}(S_2)$ are both equal to $1$, then the natural map
    \begin{equation*}
        \Isom\bigl(\mathcal{X}_1/\mathcal{O}_{K_{1},S_1},\,\mathcal{X}_2/\mathcal{O}_{K_2,S_2}\bigr)
        \longrightarrow
        \Isom^{\mathrm{GP},\Out}\Bigl(\pi^{\mathrm{\acute{e}t}}_{1}(\mathcal{X}_1)^{(\mathbb{L}_{1})},\,\pi^{\mathrm{\acute{e}t}}_{1}(\mathcal{X}_2)^{(\mathbb{L}_{2})}\Bigr)
    \end{equation*}
    is bijective, where the right-hand set denotes the set of $\Inn(\Delta_{\mathcal{X}_{2}/\mathcal{O}_{K_2,S_2}}^{\mathbb{L}_{2}})$-orbits of Galois-preserving isomorphisms (see~Definition~\ref{def:GP}).
\end{Itheorem}
At the time of writing, the authors do not know of an affirmative Neukirch--Uchida type result for rings of $S$-integers when the upper Dirichlet densities of both $\mathbb{L}(S_{1})$ and $\mathbb{L}(S_{2})$ are small, e.g., $0$.
Hence we cannot obtain the generalization of Theorem~\ref{Ithm:htrsjthtr} to such a case, at least by the methods of the present paper.
However, we can still reconstruct the geometric generic fibers of hyperbolic curves as follows:

\begin{Itheorem}[Theorem~\ref{thm:semi-abs-geometric}\ref{thm:semi-abs-geometric3-2}]\label{Ithm:semi-abs-geometric}
    We use the same notation as in Theorem~\ref{Ithm:htrsjthtr}.
    Assume that $\#\mathbb{L}_{1}\geq 2$ and that there exists a Galois-preserving isomorphism between $\pi^{\mathrm{\acute{e}t}}_{1}(\mathcal{X}_1)^{(\mathbb{L}_{1})}$ and $\pi^{\mathrm{\acute{e}t}}_{1}(\mathcal{X}_2)^{(\mathbb{L}_{2})}$.
    Then the geometric generic fibers of $\mathcal{X}_{1}$ and $\mathcal{X}_{2}$ are isomorphic as schemes.
\end{Itheorem}
The proof of Theorem~\ref{Ithm:semi-abs-geometric} is divided into three parts: (Step~1) Reconstruct the cyclotomic character from the geometrically pro-$\mathbb{L}$ fundamental group by applying Mochizuki's results; (Step~2) Reconstruct the decomposition subgroups of $G_{K,S}$ at $S$ using the result \cite[Theorem~1.1]{MR3227527}; (Step~3) Reduce to the (semi-)absolute version of the Grothendieck conjecture for hyperbolic curves over $p$-adic fields, which was proved in \cite[Theorem~3.12]{Mochizuki-Tsujimura:RIMS1974}, by a fullness result for the structure of the decomposition subgroups of $G_{K,S}$ at $S$ via \cite[Th\'eor\`eme 5.1]{MR2476781}.

\medskip

We outline the contents of each section as follows:
In Subsection~\ref{subsection1.1}, we prove Theorem~\ref{Ithm:relative_anabelian}.
In Subsection~\ref{subsection1.2}, we introduce some counterexamples to the Grothendieck conjecture based on Abhyankar's lemma.
In Subsection~\ref{subsection:Noncommutative_weight_sep_inertia}, we record the group-theoretic reconstruction of the cyclotomic character by applying Mochizuki's results.
In Subsection~\ref{subsection:proof_of_shimizu_type_result}, we prove Theorem~\ref{Ithm:htrsjthtr}.
In Subsection~\ref{subsection:proof_of_semi-abs-geometric}, we prove Theorem~\ref{Ithm:semi-abs-geometric}.

\section*{Notation and conventions}

\begin{Inotation}[The \'etale fundamental groups defined as in {\cite{MR0354651}}]\label{notation:etale}
    For every connected, locally Noetherian scheme $\mathcal{X}$, we define
    \begin{equation*}
        \Pi_{\mathcal{X}}\coloneqq \pi_1^{\mathrm{\acute{e}t}}(\mathcal{X},\ast)
    \end{equation*}
    to be the \textit{\'etale fundamental group} of $\mathcal{X}$, where $\ast:\Spec(\Omega)\rightarrow \mathcal{X}$ denotes a geometric point of $\mathcal{X}$ and $\Omega$ denotes an algebraically closed field.
    Since $\pi_1^{\mathrm{\acute{e}t}}(\mathcal{X},\ast)$ depends on the choice of base points only up to inner automorphisms, we omit the choice of base points below.

    If, moreover,  $\mathcal{X}$ is over a connected, locally Noetherian scheme $\mathcal{S}$, then we set
    \begin{equation*}
        \Delta_{\mathcal{X}/\mathcal{S}}\coloneqq \ker(\Pi_{\mathcal{X}}\to \Pi_{\mathcal{S}}),
        \ \ \ \text{and}\ \ \
        \Pi_{\mathcal{X}}^{(\Sigma)} \coloneqq \Pi_{\mathcal{X}}/\ker (\Delta_{\mathcal{X}/\mathcal{S}}\to \Delta_{\mathcal{X}/\mathcal{S}}^{\Sigma}).
    \end{equation*}
    Here, $\Sigma$ denotes a set of rational primes, and for any profinite group $G$, we denote by $G^{\Sigma}$ the maximal pro-$\Sigma$ quotient of $G$.
    For any prime $p$, we simply write $G^{p}\coloneqq G^{\{p\}}$ and $\Pi_{\mathcal{X}}^{(p)}\coloneqq\Pi_{\mathcal{X}}^{(\{p\})}$.
    We write $\widetilde{\mathcal{X}}\to\mathcal{X}$ for a connected universal Galois covering.
    We write $\tilde{\mathcal{X}}^{\Sigma}\to\mathcal{X}$ for the connected Galois covering corresponding to the closed normal subgroup $\ker(\Pi_{\mathcal{X}}\to\Pi_{\mathcal{X}}^{(\Sigma)})$; its Galois group is naturally identified, up to inner automorphism, with $\Pi_{\mathcal{X}}^{(\Sigma)}$. Note that $\tilde{\mathcal{X}}^{\Sigma}$ is a scheme by \cite[Tag 01YX]{stacks-project}.
\end{Inotation}

\begin{Inotation}[The rings of $S$-integers]\label{notation:OKS}
    For every number field $K$ (i.e., a finite extension of $\mathbb{Q}$), we denote by $\mathcal{O}_{K}$ the \textit{ring of integers} of $K$, i.e., the integral closure of the ring $\mathbb{Z}$ of rational integers in $K$.
    Let $\mathfrak{Primes}_K$ denote the set of all primes of $K$, where ``primes'' refers only to the finite (i.e., non-archimedean) primes, excluding the infinite (i.e., archimedean) ones.
    For any subset $S$ of $\mathfrak{Primes}_{K}$, we set
    \begin{equation*}
        \mathcal{O}_{K,S} \coloneqq \bigcap_{\mathfrak{p}\in\mathfrak{Primes}_K \setminus S} (\mathcal{O}_{K})_\mathfrak{p}
    \end{equation*}
    and refer to it as the \textit{ring of $S$-integers} in $K$.
    If $S=\emptyset$ (resp. $S=\mathfrak{Primes}_K$), then $\mathcal{O}_{K,S}$ coincides with $\mathcal{O}_K$ (resp. $K$).
    For any rational prime $p$, we denote by $S_p$ the set of all primes of $K$ lying over $p$.
    It is clear that $S_{p}\subset S$ if and only if $p$ is invertible in $\mathcal{O}_{K,S}$.
    We denote by $\mathbb{L}(S)$ the set of all rational primes that are invertible in $\mathcal{O}_{K,S}$.
    We write $K_{S}$ for the maximal Galois extension of $K$ unramified at all primes in $\mathfrak{Primes}_{K}\setminus S$ inside a fixed algebraic closure $\overline{K}$.
    Set
    \begin{equation*}
        G_{K,S}\coloneqq \Gal(K_{S}/K).
    \end{equation*}
    Let $\overline{\eta}$ be a geometric generic point of $\Spec(\mathcal{O}_{K,S})$ induced by the inclusion $\mathcal{O}_{K,S} \subset \overline{K}$. Then, via the natural surjection $G_K \to \pi_1^{\mathrm{\acute{e}t}}(\Spec(\mathcal{O}_{K,S}), \overline{\eta})$ induced by the inclusion $\mathcal{O}_{K,S} \hookrightarrow K$, the Galois group $G_{K,S}$ is isomorphic to $\pi_1^{\mathrm{\acute{e}t}}(\Spec(\mathcal{O}_{K,S}), \overline{\eta})$.
    If there is no risk of confusion, we identify $G_{K,S}$ with $\pi_1^{\mathrm{\acute{e}t}}(\Spec(\mathcal{O}_{K,S}), \overline{\eta})$ via this isomorphism.
\end{Inotation}

\begin{Inotation}
    For any scheme $\mathcal{S}$ and any $\mathcal{S}$-schemes $\mathcal{T}$ and $\mathcal{S}'$, we denote by $\mathcal{T}_{\mathcal{S}'}$ the base change (i.e., the fiber product) of $\mathcal{T}$ via $\mathcal{S}'\to\mathcal{S}$.
\end{Inotation}

\begin{Inotation}\label{nor:Isomset_prof}
    For $i\in \{1,2\}$, let $\Delta_{i}$, $\Pi_{i}$, and $G$ be profinite groups with the following exact sequence:
    \begin{equation*}
        1\to \Delta_{i}\to \Pi_{i}\to G\to 1
    \end{equation*}
    We denote by $\Isom_{G}(\Pi_{1},\Pi_{2})$ the set of all $G$-isomorphisms from $\Pi_{1}$ to $\Pi_{2}$, and set
    \begin{equation*}
        \Isom^{\Out}_{G}(\Pi_{1},\Pi_{2})\coloneqq \Isom_{G}(\Pi_{1},\Pi_{2})/\Inn(\Delta_{2}),
    \end{equation*}
    where $\Inn(\Delta_{2})$ denotes the subgroup of
    $\Isom_{G}(\Pi_{2}, \Pi_{2})$ consisting of inner automorphisms given by conjugation by elements of $\Delta_{2}$.
\end{Inotation}

\begin{Inotation}\label{nor:Isomset_sch}
If $f_i:\mathcal{X}_i\to\mathcal{S}_i$ are morphisms of schemes for $i=1,2$, we write
\[
    \Isom(\mathcal{X}_1/\mathcal{S}_1,\mathcal{X}_2/\mathcal{S}_2)
\]
for the set of pairs $(\phi_{\mathcal X},\phi_{\mathcal S})$ consisting of isomorphisms $\phi_{\mathcal X}:\mathcal{X}_1\xrightarrow{\sim}\mathcal{X}_2$ and $\phi_{\mathcal S}:\mathcal{S}_1\xrightarrow{\sim}\mathcal{S}_2$ such that $f_2\circ\phi_{\mathcal X}=\phi_{\mathcal S}\circ f_1$.
\end{Inotation}

\section{Relative version of the Grothendieck conjecture over certain arithmetic schemes}\label{section1}
In this section, we prove the relative version of the Grothendieck conjecture for hyperbolic curves over certain arithmetic schemes when there exists a rational prime that is invertible on the base scheme in Subsection~\ref{subsection1.1} and give some counterexamples to the Grothendieck conjecture when no prime is invertible on the base scheme in~Subsection~\ref{subsection1.2}.

\subsection{Proof of the relative Grothendieck conjecture when there exists an invertible rational prime}\label{subsection1.1}

\subsubsection{}
We first define a smooth curve over a scheme and establish a basic bijection between the automorphisms of a hyperbolic curve over a connected, normal, locally Noetherian scheme and those of its generic fiber (see~Proposition~\ref{prop:Aut-restriction} below).
Note that a connected, normal, locally Noetherian scheme is integral.
A proposition similar to the following can be found in several places (e.g., in \cite[Theorem~1.1(1)]{MR2172341} and \cite[Proposition~9.3.13, Example~9.3.15]{MR1917232}).

\begin{definition}\label{smcurvedef}
    A \textit{smooth curve} (of type $(g, r)$) over a scheme $\mathcal{S}$ is a pair $(\mathcal{X}^{\mathrm{cpt}}, \mathcal{D})$ consisting of an $\mathcal{S}$-scheme $\mathcal{X}^{\mathrm{cpt}}$ and a (possibly empty) closed subscheme $\mathcal{D} \subset \mathcal{X}^{\mathrm{cpt}}$ with the following properties:
    \begin{enumerate}[(i)]
        \item
              $\mathcal{X}^{\mathrm{cpt}}$ is smooth, proper, and of relative dimension one over $\mathcal{S}$.
        \item
              For any geometric point $\overline{s}$ of $\mathcal{S}$, the geometric fiber $\mathcal{X}^{\mathrm{cpt}}_{\overline{s}}$ is connected and satisfies
              \begin{equation*}
                  \dim H^1(\mathcal{X}^{\mathrm{cpt}}_{\overline{s}}, \mathcal{O}_{\mathcal{X}^{\mathrm{cpt}}_{\overline{s}}}) = g.
              \end{equation*}
        \item
              The composition $\mathcal{D} \hookrightarrow \mathcal{X}^{\mathrm{cpt}} \to \mathcal{S}$ is finite \'etale of degree $r$.
    \end{enumerate}
    A smooth curve of type $(g, r)$ is \textit{hyperbolic} if $2-2g-r<0$, i.e., $(g,r)\notin\{(0,0), (0,1), (0,2), (1,0)\}$.
    Set $\mathcal{X} \coloneqq \mathcal{X}^{\mathrm{cpt}} \setminus \mathcal{D}$ for the complement of $\mathcal{D}$ in $\mathcal{X}^{\mathrm{cpt}}$.
    Note that if $\mathcal{S}$ is normal and locally Noetherian, then the smooth compactification of $\mathcal{X}$ is uniquely determined by $\mathcal{X}$ up to unique isomorphism (see~\cite[p.230]{MR2215441}).
    Therefore, when there is no risk of confusion, we also refer to $\mathcal{X}$ as a smooth curve (of type $(g,r)$) over $\mathcal{S}$.
\end{definition}

\begin{lemma}\label{basicconditioneq}
    Let $F$ be an algebraically closed field, and let $\mathbb{L}$ be a nonempty set of rational primes that are invertible in $F$.
    Let $\overline{X}$ be a smooth curve of type $(g,r)$ over $F$.
    Then the following hold:
    \begin{enumerate}[(1)]
        \item \label{basicconditioneq1}
              $(g,r)\notin\{(0,0),(0,1)\}$ if and only if $\Delta^{\mathrm{ab},\mathbb{L}}_{\overline{X}/F}$ is nontrivial.
        \item\label{basicconditioneq2}
              $\overline{X}$ is hyperbolic if and only if $\Delta^{\mathbb{L}}_{\overline{X}/F}$ is not abelian.
    \end{enumerate}
\end{lemma}

\begin{proof}
    We have that $\Delta^{\mathbb{L}}_{\overline{X}/F}$ is a free pro-$\mathbb{L}$ group (resp. a pro-$\mathbb{L}$ surface group) of rank $2g+r-1$ (resp. $g$) when $r\geq 1$ (resp. $r=0$) (see \cite[Expos\'e XIII, Corollaire 2.12]{MR0354651}), and hence $\Delta^{\mathrm{ab},\mathbb{L}}_{\overline{X}/F}$ is a free pro-$\mathbb{L}$ abelian group of rank $2g+r-\varepsilon$, where $\varepsilon$ denotes $1$ (resp. $0$).
    Thus, the assertion \ref{basicconditioneq1} follows.
    The assertion \ref{basicconditioneq2} follows from the direct calculation.
\end{proof}

\begin{lemma}\label{lem:basic_injectivity_}
    Let $\mathcal{S}$ be an integral scheme with generic point $\eta$.
    Let $\mathcal{X}$ be a separated $\mathcal{S}$-scheme and $\mathcal{Y}$ a flat $\mathcal{S}$-scheme.
    If two $\mathcal{S}$-morphisms $f,g:\mathcal{Y}\to \mathcal{X}$ coincide on the generic fiber $\mathcal{Y}_\eta$, i.e., $f|_{\mathcal{Y}_\eta}=g|_{\mathcal{Y}_\eta}$, then $f=g$.
\end{lemma}

\begin{proof}
    Let $\mathrm{Eq}(f,g)$ be the equalizer of $f$ and $g$.
    Since $\mathcal{X}/\mathcal{S}$ is separated, $\mathrm{Eq}(f,g)$ is a closed subscheme of $\mathcal{Y}$.
    As $f|_{\mathcal{Y}_\eta}=g|_{\mathcal{Y}_\eta}$, the morphism $\mathcal{Y}_{\eta}\to \mathcal{Y}$ factors through $\mathrm{Eq}(f,g)$.
    Since $\eta\to \mathcal{S}$ is quasi-compact and $\mathcal{Y}$ is flat over $\mathcal{S}$, $\mathcal{Y}_{\eta}$ is schematically dense in $\mathcal{Y}$ (see~\cite[Th\'eor\`eme 11.10.5]{MR0217086}).
    Therefore, $\mathrm{Eq}(f,g)=\mathcal{Y}$ follows.
    This completes the proof.
\end{proof}

\begin{proposition}\label{prop:Aut-restriction}
    Let $\mathcal{S}$ be a connected, normal, locally Noetherian scheme with generic point $\eta$.
    Let $\mathcal{X}$ be a hyperbolic curve over $\mathcal{S}$.
    Then the natural restriction map
    \begin{equation*}
        \Aut_{\mathcal{S}}(\mathcal{X})
        \xrightarrow{\ \sim\ }
        \Aut_{K(\mathcal{S})}(\mathcal{X}_\eta); \qquad f\longmapsto f|_{\mathcal{X}_\eta}
    \end{equation*}
    is bijective.
\end{proposition}

\begin{proof}
    Let $(g,r)$ be the type of $\mathcal{X}$.
    By the note in Definition~\ref{smcurvedef}, the smooth compactifications of $\mathcal{X}$ and $\mathcal{X}_\eta$ are uniquely determined up to unique isomorphism.
    Hence it suffices to prove that the restriction map
    \begin{equation*}
        \Aut_{\mathcal{S}}\bigl((\mathcal{X}^{\mathrm{cpt}},\mathcal{D})\bigr)
        \longrightarrow
        \Aut_{K(\mathcal{S})}\bigl((\mathcal{X}_\eta^{\mathrm{cpt}},\mathcal{D}_\eta)\bigr)
    \end{equation*}
    is bijective.
    The injectivity follows from Lemma~\ref{lem:basic_injectivity_}.
    Next, we show surjectivity.
    Let $\mathcal{M}_{g,[r]}$ be the moduli stack of smooth curves of type $(g,r)$.
    It is a separated Deligne--Mumford stack (see~\cite[Theorem~2.7]{MR0702953}). (Note that the term ``algebraic stack'' in \cite{MR0702953} is used in the sense of \cite{MR0262240}.
    We refer to it as a \textit{Deligne--Mumford stack} here.)
    In particular, the diagonal morphism is representable, unramified, and proper.
    Let
    \begin{equation*}
        \underline{\Aut}_{\mathcal{S}}
        \coloneqq
        \underline{\Isom}_{\mathcal{S}}\bigl((\mathcal{X}^{\mathrm{cpt}},\mathcal{D}),(\mathcal{X}^{\mathrm{cpt}},\mathcal{D})\bigr)
    \end{equation*}
    be the relative automorphism functor of the pointed smooth curves.
    By representability of the diagonal morphism, the stack $\underline{\Aut}_{\mathcal{S}}$ is represented by an $\mathcal{S}$-scheme that is proper and unramified over $\mathcal{S}$.
    Moreover, since each geometric fiber is hyperbolic, its automorphism group is finite.
    Hence $\underline{\Aut}_{\mathcal{S}}\to \mathcal{S}$ is quasi-finite, and therefore finite by \cite[Tag 02OG]{stacks-project}.

    Let $\psi\in \Aut_{K(\mathcal{S})}\bigl((\mathcal{X}_\eta^{\mathrm{cpt}},\mathcal{D}_\eta)\bigr)$.
    Via the preceding identification, $\psi$ determines a $K(\mathcal{S})$-point $\sigma_\eta:\Spec(K(\mathcal{S}))\to \underline{\Aut}_{\mathcal{S}}$ and the composite morphism $\Spec(K(\mathcal{S}))\to \underline{\Aut}_{\mathcal{S}}\to \mathcal{S}$ coincides with $\eta$.
    Let $Z\subset \underline{\Aut}_{\mathcal{S}}$ be the scheme-theoretic closure of the image of $\sigma_\eta$.
    Then $Z\to \mathcal{S}$ is finite and birational.
    Since $\mathcal{S}$ is connected and normal, a finite birational morphism onto $\mathcal{S}$ is an isomorphism (see~\cite[Corollaire 4.4.9]{MR0217085} or \cite[Tag 0AB1]{stacks-project}).
    Thus $Z\xrightarrow{\sim} \mathcal{S}$, and hence $\sigma_\eta$ extends uniquely to an $\mathcal{S}$-section $\mathcal{S}\to \underline{\Aut}_{\mathcal{S}}$, i.e., $\psi$ extends uniquely to an element of $\Aut_{\mathcal{S}}\bigl((\mathcal{X}^{\mathrm{cpt}},\mathcal{D})\bigr)$.
\end{proof}

\subsubsection{}
Next, we recall a basic result for the homotopy exact sequence of \'etale fundamental groups of smooth curves.
The right-exactness of the homotopy exact sequence can be found in \cite[Expos\'e XIII, Lemme 4.2, Exemples 4.4]{MR0354651}.
Moreover, if it admits a section, then it is exact (see~\cite[Expos\'e XIII, Proposition 4.3, Exemples 4.4]{MR0354651}).
J.~Stix \cite{MR2172341} established exactness in the case of hyperbolic curves as follows:

\begin{proposition}[{\cite[Proposition 2.3]{MR2172341}}]\label{lem:stix_exact}
    Let $\mathcal{S}$ be a connected, locally Noetherian scheme, and set $s\in \mathcal{S}$.
    Let $\kappa(s)$ be the residue field of $s$.
    Let $\mathbb{L}$ be a set of rational primes that are invertible on $\mathcal{S}$.
    Let $\mathcal{X}$ be a hyperbolic curve over $\mathcal{S}$ with special fiber $\mathcal{X}_{s}$ at $s$.
    Then the natural morphism $\Delta_{\mathcal{X}_{s}/\kappa(s)}^{\mathbb{L}}\rightarrow \Delta_{\mathcal{X}/\mathcal{S}}^{\mathbb{L}}$ is an isomorphism.
    In other words, the sequence
    \begin{equation*}
        1\to \Delta_{\mathcal{X}_{s}/\kappa(s)}^{\mathbb{L}}
        \to \Pi_{\mathcal{X}}^{(\mathbb{L})}
        \to \Pi_{\mathcal{S}}
        \to 1
    \end{equation*}
    of natural morphisms is exact.
\end{proposition}

\subsubsection{}
Next, we prove the following lemma in group theory, which is well-known to experts.
This statement can be found, for instance, in \cite[Section 7]{MR1478817}.
However, we could not find a written proof, and hence we include one here.

\begin{lemma}\label{lem:fiber_product}
    Let $\Delta$, $\Pi$, $G$, $I$, $A$, and $O$ be profinite groups with the commutative diagram of two horizontal exact sequences
    \begin{equation*}
        \vcenter{
            \xymatrix{
                1\ar[r]&\Delta\ar[r]\ar[d]&\Pi\ar[r]\ar[d]&G\ar[r]\ar[d]&1\\
                1\ar[r]&I\ar[r]&A\ar[r]&O\ar[r]&1.
            }
        }
    \end{equation*}
    Assume that the above left-hand vertical morphism $\Delta\to I$ is an isomorphism.
    Then the natural morphism from $\Pi$ to the fiber product $A\times_{O}G$ is an isomorphism.
\end{lemma}

\begin{proof}
    Consider the commutative diagram with exact rows
    \begin{equation}\label{fiberproduct_of_groups}
        \vcenter{
        \xymatrix{
        1\ar[r]&\Delta\ar[r]\ar[d]^{\rho_{2}}&\Pi\ar[r]\ar[d]^{\rho_{1}}&G\ar[r]\ar@{=}[d]&1\\
        1\ar[r]&K\ar[r]\ar[d]^{\rho_{3}}&A\times_{O}G\ar[r]\ar[d]&G\ar[d]\ar[r]&1\\
        1\ar[r]&I\ar[r]&A\ar[r]&O\ar[r]&1.
        }
        }
    \end{equation}
    Here, $K$ denotes the kernel of the natural projection $A \times_{O} G \to G$, and the morphisms $\rho_1$, $\rho_2$, and $\rho_3$ denote the natural morphisms.
    The surjectivity of the morphism $A \times_{O} G \to G$ follows from the commutativity of the upper right square in \eqref{fiberproduct_of_groups}.
    The bijectivity of $\rho_3 \circ \rho_2$ implies that $\rho_3$ is surjective.
    Let $\rho_4 \colon \ker(\rho_3) \to A \times_{O} G$ be the morphism induced by the fiber product diagram
    \begin{equation*}
        \vcenter{
        \xymatrix{
        \ker(\rho_{3})\ar@{.>}[rd]_{\rho_{4}}\ar@{^{(}->}[d]\ar[drr]^{\text{zero map}}\\
        K\ar[d]^{\rho_{3}}&A\times_{O}G\ar[r]\ar[d]&G\ar[d]\\
        I\ar[r]&A\ar[r]&O.
        }
        }
    \end{equation*}
    Then, since $\ker(\rho_{3})\to A$ and $\ker(\rho_{3})\to G$ are zero maps, $\rho_{4}$ is also zero.
    This implies that $\rho_{3}$ is bijective, and hence $\rho_{2}$ is also bijective.
    Thus, by the Five Lemma, $\rho_{1}$ is bijective.
    This completes the proof.
\end{proof}

\subsubsection{}
Using the above results and the Grothendieck conjecture for hyperbolic curves over generalized sub-$p$-adic fields \cite[Theorem~4.12]{MR2012215}, we obtain the first main theorem of the present paper:

\begin{definition}[{\cite[Definition 4.11]{MR2012215}}]\label{def:gen-sub-p-adic}
    Let $p$ be a rational prime.
    We say that a field $K$ is \textit{generalized sub-$p$-adic} if $K$ embeds as a subfield of a finitely generated extension of the fraction field of the ring of the Witt vectors of $\overline{\mathbb{F}}_{p}$ (see, e.g., Example~\ref{ex:conditiona}\ref{ex:conditiona1} below).
\end{definition}

\begin{theorem}\label{thm:relative_anabelian}
    Let $\mathcal{X}_{1}$ and $\mathcal{X}_{2}$ be smooth curves over a scheme $\mathcal{S}$.
    Let $\mathbb{L}$ be a set of rational primes that are invertible on $\mathcal{S}$.
    Assume the following conditions:
    \begin{enumerate}[(a)]
        \item\label{thm:relative_anabelian_condition0}
              $\mathcal{S}$ is a connected, normal, locally Noetherian scheme whose function field is a generalized sub-$p$-adic field for some invertible rational prime $p\in \mathbb{L}$ (see, e.g., Example~\ref{ex:conditiona} below).
        \item\label{thm:relative_anabelian_condition1}
              At least one of $\mathcal{X}_{1}$ and $\mathcal{X}_{2}$ is hyperbolic.
    \end{enumerate}
    Then the natural map
    \begin{equation*}
        \Isom_{\mathcal{S}}(\mathcal{X}_1,\,\mathcal{X}_2)
        \longrightarrow
        \Isom_{\Pi_{\mathcal{S}}}^{\Out}\bigl(\Pi_{\mathcal{X}_1}^{(\mathbb{L})},\,\Pi_{\mathcal{X}_2}^{(\mathbb{L})}\bigr)
    \end{equation*}
    is bijective, where the set $\Isom_{\ast}^{\Out}(\ast,\ast)$ is defined as in Notation~\ref{nor:Isomset_prof}.
\end{theorem}

\begin{proof}
    If $\Isom_{\Pi_{\mathcal{S}}}(\Pi_{\mathcal{X}_{1}}^{(\mathbb{L})}, \Pi_{\mathcal{X}_{2}}^{(\mathbb{L})}) = \emptyset$, then the assertion is clear.
    Hence, we may assume that $\Isom_{\Pi_{\mathcal{S}}}(\Pi_{\mathcal{X}_{1}}^{(\mathbb{L})}, \Pi_{\mathcal{X}_{2}}^{(\mathbb{L})}) \neq \emptyset$.
    Fix a $\Pi_{\mathcal{S}}$-isomorphism $\Phi:\Pi_{\mathcal{X}_1}^{(\mathbb{L})}\xrightarrow{\sim}\Pi_{\mathcal{X}_2}^{(\mathbb{L})}$.
    Let $\eta$ be the generic point of $\mathcal{S}$ and set $X_{i}\coloneqq \mathcal{X}_{i,\eta}$.
    Let $K(\mathcal{S})$ be the function field of $\mathcal{S}$.
    Then we have the following commutative diagram with exact rows:
    \begin{equation}\label{thm:relative_anabelian_eq_22}
        \vcenter{
            \xymatrix{
                1\ar[r] & \Delta_{X_{i}/K(\mathcal{S})}^{\mathbb{L}} \ar[r] \ar[d] & \Pi_{X_{i}}^{(\mathbb{L})} \ar[r] \ar@{->>}[d] & G_{K(\mathcal{S})} \ar[r] \ar@{->>}[d] & 1 \\
                1\ar[r] & \Delta_{\mathcal{X}_{i}/\mathcal{S}}^{\mathbb{L}} \ar[r] & \Pi_{\mathcal{X}_{i}}^{(\mathbb{L})} \ar[r] & \Pi_{\mathcal{S}} \ar[r] & 1 \\
            }
        }
    \end{equation}
    Here, the left-hand vertical arrow is the isomorphism discussed in Proposition~\ref{lem:stix_exact}.
    Hence $\Phi$ induces an isomorphism $\Delta_{X_1/K(\mathcal{S})}^{\mathbb{L}}\xrightarrow{\sim}\Delta_{X_2/K(\mathcal{S})}^{\mathbb{L}}$.
    By definition, $\mathcal{X}_{i}$ is hyperbolic if and only if $X_{i}$ is hyperbolic.
    Hence Lemma~\ref{basicconditioneq}\ref{basicconditioneq2} and the hypothesis \ref{thm:relative_anabelian_condition1} imply that both smooth curves $\mathcal{X}_{1}$ and $\mathcal{X}_{2}$ are hyperbolic.
    Applying Lemma~\ref{lem:fiber_product} to the diagram \eqref{thm:relative_anabelian_eq_22}, the isomorphism $\Phi$ induces a $G_{K(\mathcal{S})}$-isomorphism $\Pi_{X_{1}}^{(\mathbb{L})}\xrightarrow{\sim}\Pi_{X_{2}}^{(\mathbb{L})}$.
    Equivalently, we obtain the following map:
    \begin{equation*}
        \Isom_{\Pi_{\mathcal{S}}}\bigl(\Pi_{\mathcal{X}_{1}}^{(\mathbb{L})},\,\Pi_{\mathcal{X}_{2}}^{(\mathbb{L})}\bigr)
        \rightarrow
        \Isom_{G_{K(\mathcal{S})}}\bigl(\Pi_{X_{1}}^{(\mathbb{L})},\,\Pi_{X_{2}}^{(\mathbb{L})}\bigr)
    \end{equation*}
    Since $\Pi_{X_{i}}^{(\mathbb{L})}\to\Pi_{\mathcal{X}_{i}}^{(\mathbb{L})}$ is surjective, the above map is injective.
    Then we have the following commutative diagram of the natural maps:
    \begin{equation}\label{commutativi_diag_relGC}
        \vcenter{
        \xymatrix{
        \Isom_{K(\mathcal{S})}(X_{1},\,X_{2})\ar[r]^-{\sim}&\Isom^{\Out}_{G_{K(\mathcal{S})}}\bigl(\Pi_{X_{1}}^{(\mathbb{L})},\,\Pi_{X_{2}}^{(\mathbb{L})}\bigr)\\
        \Isom_{\mathcal{S}}\bigl(\mathcal{X}_1,\,\mathcal{X}_2\bigr)\ar[u]^-{\cong}\ar[r]&\Isom^{\Out}_{\Pi_{\mathcal{S}}}\bigl(\Pi_{\mathcal{X}_{1}}^{(\mathbb{L})},\,\Pi_{\mathcal{X}_{2}}^{(\mathbb{L})}\bigr)\ar@{^{(}->}[u]
        }
        }
    \end{equation}
    Here, the bijectivity of the left-hand vertical arrow follows from
    Proposition~\ref{prop:Aut-restriction}, and the bijectivity of the upper horizontal arrow follows from \cite[Theorem~4.12]{MR2012215} and hypotheses \ref{thm:relative_anabelian_condition0} and \ref{thm:relative_anabelian_condition1}.
    Therefore, the bijectivity of the lower horizontal map follows from the commutativity of the diagram \eqref{commutativi_diag_relGC}.
    This completes the proof.
\end{proof}

\begin{example}\label{ex:conditiona}
    The condition \ref{thm:relative_anabelian_condition0} in Theorem~\ref{thm:relative_anabelian} is satisfied in the following cases:
    \begin{enumerate}[(1)]
        \item\label{ex:conditiona1}
              The spectrum of every generalized sub-$p$-adic field satisfies the condition \ref{thm:relative_anabelian_condition0}.
              For instance, for any rational prime $p$, every number field (or more generally, every field finitely generated over $\mathbb{Q}$) is a generalized sub-$p$-adic field.
              Moreover, every $p$-adic local field is also a generalized sub-$p$-adic field (see~\cite[Definition 15.4]{MR1720187}).
        \item
              Recall that a \textit{Dedekind scheme} is an integral, regular, Noetherian scheme of dimension one.
              Then every Dedekind scheme whose function field is a generalized sub-$p$-adic field for some invertible rational prime $p$ satisfies the condition \ref{thm:relative_anabelian_condition0} (e.g., rings of $S$-integers in which $p$ is invertible, see~Notation~\ref{notation:OKS}).
        \item
              If a scheme $\mathcal{S}$ satisfies the condition \ref{thm:relative_anabelian_condition0}, then any connected, normal $\mathcal{S}$-scheme of finite type also satisfies the condition \ref{thm:relative_anabelian_condition0} (e.g., $\mathbb{P}^{n}_{\mathcal{S}}$ and $\mathbb{A}^{n}_{\mathcal{S}}$ for $n\in\mathbb{Z}_{\geq 1}$).
    \end{enumerate}
\end{example}

\begin{corollary}\label{cor:relative_anabelian}
    We keep the notation of Theorem~\ref{thm:relative_anabelian} and use the same hypotheses \ref{thm:relative_anabelian_condition0} and \ref{thm:relative_anabelian_condition1}.
    Then the natural map
    \begin{equation*}
        \Isom_{\widetilde{\mathcal{S}}/\mathcal{S}}\!\bigl(\tilde{\mathcal{X}}_{1}^{\mathbb{L}}/\mathcal{X}_{1}, \tilde{\mathcal{X}}_{2}^{\mathbb{L}}/\mathcal{X}_{2}\bigr)
        \longrightarrow
        \Isom_{\Pi_{\mathcal{S}}}\bigl(\Pi_{\mathcal{X}_{1}}^{(\mathbb{L})}, \Pi_{\mathcal{X}_{2}}^{(\mathbb{L})}\bigr)
    \end{equation*}
    is bijective, where $\tilde{\mathcal{X}}_{i}^{\mathbb{L}}$ and $\widetilde{\mathcal{S}}$ are chosen as in Notation~\ref{notation:etale}.
\end{corollary}

\begin{proof}
    Consider the commutative diagram with exact rows
    \begin{equation*}
        \xymatrix@C=14pt{
        1\ar[r]&
        \Aut_{\mathcal{X}_{2,\tilde{\mathcal{S}}}}\bigl(\tilde{\mathcal{X}}_2^{\mathbb{L}}\bigr)\ar@{->>}[d]^{\theta}\ar[r]&
        \Isom_{\widetilde{\mathcal{S}}/\mathcal{S}}\bigl(\tilde{\mathcal{X}}_1^{\mathbb{L}}/\mathcal{X}_1,\tilde{\mathcal{X}}_2^{\mathbb{L}}/\mathcal{X}_2\bigr)\ar[d]\ar[r]&
        \Isom_{\mathcal{S}}(\mathcal{X}_1,\mathcal{X}_2)\ar[d]^{\Psi}\ar[r]&1\\
        1\ar[r]&
        \Inn(\Delta_{\mathcal{X}_2/\mathcal{S}}^{\mathbb{L}})\ar[r]&
        \Isom_{\Pi_{\mathcal{S}}}\bigl(\Pi_{\mathcal{X}_1}^{(\mathbb{L})},\Pi_{\mathcal{X}_2}^{(\mathbb{L})}\bigr)\ar[r]&
        \Isom_{\Pi_{\mathcal{S}}}^{\Out}\bigl(\Pi_{\mathcal{X}_1}^{(\mathbb{L})},\Pi_{\mathcal{X}_2}^{(\mathbb{L})}\bigr)\ar[r] &1.
        }
    \end{equation*}
    The exactness of the upper sequence follows from Proposition~\ref{prop:Aut-restriction} and \cite[Lemma 6.2]{MR1478817}.
    By Theorem~\ref{thm:relative_anabelian}, the right-hand vertical arrow $\Psi$ is bijective.
    Let $\eta$ be the generic point of $\mathcal{S}$ and set $X_{i}\coloneqq \mathcal{X}_{i,\eta}$.
    Since $\Aut_{\mathcal{X}_{2,\tilde{\mathcal{S}}}}(\tilde{\mathcal{X}}_2^{\mathbb{L}})\cong \Delta_{\mathcal{X}_2/\mathcal{S}}^{\mathbb{L}}\cong \Delta_{X_{2}/K(\mathcal{S})}^{\mathbb{L}}$ by Proposition~\ref{lem:stix_exact} and the fact that $\Delta_{X_{2}/K(\mathcal{S})}^{\mathbb{L}}$ is center-free (see~\cite[Proposition 1.11]{MR1478817}), the left-hand vertical arrow $\theta$ is bijective.
    Therefore the middle vertical arrow is also bijective by the Five Lemma.
    This completes the proof.
\end{proof}

\subsubsection{}
We obtain pro-$\Sigma$ versions of Theorem~\ref{thm:relative_anabelian} and Corollary~\ref{cor:relative_anabelian} for a certain set $\Sigma$ of rational primes.
In particular, the profinite versions of Theorem~\ref{thm:relative_anabelian} and Corollary~\ref{cor:relative_anabelian} also hold, which is the basic setting in anabelian geometry.
Although the following argument of the generalization from the pro-$\mathbb{L}$ version to the pro-$\Sigma$ version is standard in anabelian geometry, we provide the details here for the convenience of the reader.

\begin{corollary}\label{cor:relative_anabelian_conditionsigma}
    Let $\mathcal{X}_{1}$ and $\mathcal{X}_{2}$ be smooth curves over a scheme $\mathcal{S}$.
    Let $\Sigma$ be a set of rational primes.
    Assume the following conditions:
    \begin{enumerate}[(a)]
        \item\label{cor:relative_anabelian_conditionsigma0}
              $\mathcal{S}$ is a connected, normal, locally Noetherian scheme whose function field is a generalized sub-$p$-adic field for some $p\in \Sigma$ that is invertible on $\mathcal{S}$.
        \item\label{cor:relative_anabelian_conditionsigma1}
              At least one of $\mathcal{X}_{1}$ and $\mathcal{X}_{2}$ is hyperbolic.
    \end{enumerate}
    Then the following hold:
    \begin{enumerate}[(1)]
        \item\label{cor:relative_anabelian_conditionsigmalem1}
              The natural map
              \begin{equation*}
                  \Isom_{\widetilde{\mathcal{S}}/\mathcal{S}}(\tilde{\mathcal{X}}^{\Sigma}_{1}/\mathcal{X}_{1},\,\tilde{\mathcal{X}}^{\Sigma}_{2}/\mathcal{X}_{2})
                  \to
                  \Isom_{\Pi_{\mathcal{S}}}(\Pi_{\mathcal{X}_{1}}^{(\Sigma)},\,\Pi_{\mathcal{X}_{2}}^{(\Sigma)})
              \end{equation*}
              is bijective, where $\tilde{\mathcal{X}}_{i}^{\Sigma}$ and $\widetilde{\mathcal{S}}$ are chosen as in Notation~\ref{notation:etale}.
        \item\label{cor:relative_anabelian_conditionsigmalem3}
              The natural map
              \begin{equation*}
                  \Isom_{\mathcal{S}}(\mathcal{X}_1,\,\mathcal{X}_2)
                  \longrightarrow
                  \Isom_{\Pi_{\mathcal{S}}}^{\Out}\bigl(\Pi_{\mathcal{X}_1}^{(\Sigma)},\,\Pi_{\mathcal{X}_2}^{(\Sigma)}\bigr)
              \end{equation*}
              is bijective, where the set $\Isom_{\ast}^{\Out}(\ast,\ast)$ is defined as in Notation~\ref{nor:Isomset_prof}.
    \end{enumerate}
\end{corollary}

\begin{proof}
    Let $i$ range over $\{1,2\}$.
    First, we show \ref{cor:relative_anabelian_conditionsigmalem1}.
    If $\Isom_{\Pi_{\mathcal{S}}}(\Pi_{\mathcal{X}_{1}}^{(\Sigma)}, \Pi_{\mathcal{X}_{2}}^{(\Sigma)}) = \emptyset$, then the assertion is clear.
    Hence, we may assume that $\Isom_{\Pi_{\mathcal{S}}}(\Pi_{\mathcal{X}_{1}}^{(\Sigma)}, \Pi_{\mathcal{X}_{2}}^{(\Sigma)}) \neq \emptyset$.
    Let $\Phi:\Pi_{\mathcal{X}_1}^{(\Sigma)}\xrightarrow{\ \sim\ }\Pi_{\mathcal{X}_2}^{(\Sigma)}$ be a $\Pi_{\mathcal{S}}$-isomorphism.
    Let $\eta$ be the generic point of $\mathcal{S}$ and set $X_{i}\coloneqq \mathcal{X}_{i,\eta}$.
    Let $K(\mathcal{S})$ be the function field of $\mathcal{S}$.
    By Proposition~\ref{lem:stix_exact}, $\Phi$ induces an isomorphism $\Delta_{X_1/K(\mathcal{S})}^{\Sigma}\xrightarrow{\sim}\Delta_{X_2/K(\mathcal{S})}^{\Sigma}$.
    Hence Lemma~\ref{basicconditioneq}\ref{basicconditioneq2} and hypothesis \ref{cor:relative_anabelian_conditionsigma1} imply that both smooth curves $\mathcal{X}_{1}$ and $\mathcal{X}_{2}$ are hyperbolic.
    Let $N_1\subset\Pi_{\mathcal{X}_1}^{(\Sigma)}$ be an open subgroup and set $N_2\coloneqq \Phi(N_1)$.
    Let $\mathcal{Y}_i\to\mathcal{X}_i$ be the connected finite \'etale covering corresponding to $N_i$.
    Let $\mathcal{S}'\to \mathcal{S}$ be the connected finite \'etale covering corresponding to the image of $N_i$ in $\Pi_{\mathcal{S}}$.
    Let $p$ be the rational prime as in hypothesis \ref{cor:relative_anabelian_conditionsigma0}.
    Then we have the commutative diagram with exact rows
    \begin{equation*}
        \xymatrix{
        1\ar[r]&\ar[d]^{\cong}\Delta_{\mathcal{Y}_{i}/\mathcal{S}'}^{p}\ar[r]&\ar[d]\Pi_{\mathcal{Y}_{i}}^{(p)}\ar[r]&\Pi_{\mathcal{S}'}\ar[r]\ar@{=}[d]&1\\
        1\ar[r]&\ker(N_{i}\to \Pi_{\mathcal{S}'})^{p}\ar[r]&N_{i}^{(p)}\ar[r]&\Pi_{\mathcal{S}'}\ar[r]&1,
        }
    \end{equation*}
    where $N_{i}^{(p)}$ denotes the quotient
    \[
        N_i/\ker\Bigl(\ker(N_i\to\Pi_{\mathcal{S}'})\to\ker(N_i\to\Pi_{\mathcal{S}'})^{p}\Bigr),
    \]
    so that it fits into the lower exact sequence.
    Hence $\Phi$ induces a $\Pi_{\mathcal{S}'}$-isomorphism $\Pi_{\mathcal{Y}_{1}}^{(p)}\xrightarrow{\sim}\Pi_{\mathcal{Y}_{2}}^{(p)}$.
    Then Corollary~\ref{cor:relative_anabelian} implies that there exists a unique element
    \begin{equation*}
        (\tilde\phi_{\mathcal{Y}_{1}}^{p},\,\phi_{\mathcal{Y}_{1}})
        \in
        \Isom_{\widetilde{\mathcal{S}}/\mathcal{S}}\bigl(\tilde{\mathcal{Y}}_1^{p}/\mathcal{Y}_1,\,\tilde{\mathcal{Y}}_2^{p}/\mathcal{Y}_2\bigr)
    \end{equation*}
    that induces the $\Pi_{\mathcal{S}'}$-isomorphism $\Pi_{\mathcal{Y}_{1}}^{(p)}\xrightarrow{\sim}\Pi_{\mathcal{Y}_{2}}^{(p)}$.
    By running over all $N_{1}$ as above and using the above uniqueness at each finite level, these data are compatible under pullback.
    Hence, we obtain a unique element
    \begin{equation*}
        (\tilde\phi,\phi_{\mathcal{X}_{1}})
        \in
        \Isom_{\widetilde{\mathcal{S}}/\mathcal{S}}\bigl(\tilde{\mathcal{X}}_1^{\Sigma}/\mathcal{X}_1,\,\tilde{\mathcal{X}}_2^{\Sigma}/\mathcal{X}_2\bigr)
    \end{equation*}
    such that the $\Pi_{\mathcal{S}}$-isomorphism $\Pi_{\mathcal{X}_1}^{(\Sigma)}\xrightarrow{\sim}\Pi_{\mathcal{X}_2}^{(\Sigma)}$ induced by this element coincides with $\Phi$, where $\tilde{\phi}$ is induced by the inverse system of the morphisms $\phi_{\mathcal{Y}_{1}}$.
    This completes the proof of \ref{cor:relative_anabelian_conditionsigmalem1}.

    Next, we show \ref{cor:relative_anabelian_conditionsigmalem3}.
    Consider the commutative diagram with exact rows
    \begin{equation*}
        \xymatrix@C=14pt{
        1\ar[r]&
        \Aut_{\mathcal{X}_{2,\tilde{\mathcal{S}}}}\bigl(\tilde{\mathcal{X}}_2^{\Sigma}\bigr)\ar@{->>}[d]\ar[r]&
        \Isom_{\widetilde{\mathcal{S}}/\mathcal{S}}\bigl(\tilde{\mathcal{X}}_1^{\Sigma}/\mathcal{X}_1,\tilde{\mathcal{X}}_2^{\Sigma}/\mathcal{X}_2\bigr)\ar[d]\ar[r]&
        \Isom_{\mathcal{S}}(\mathcal{X}_1,\mathcal{X}_2)\ar[d]\ar[r]&1\\
        1\ar[r]&
        \Inn(\Delta_{\mathcal{X}_2/\mathcal{S}}^{\Sigma})\ar[r]&
        \Isom_{\Pi_{\mathcal{S}}}\bigl(\Pi_{\mathcal{X}_1}^{(\Sigma)},\Pi_{\mathcal{X}_2}^{(\Sigma)}\bigr)\ar[r]&
        \Isom_{\Pi_{\mathcal{S}}}^{\Out}\bigl(\Pi_{\mathcal{X}_1}^{(\Sigma)},\Pi_{\mathcal{X}_2}^{(\Sigma)}\bigr)\ar[r] &1,
        }
    \end{equation*}
    where the exactness of the upper sequence follows from Proposition~\ref{prop:Aut-restriction} and \cite[Lemma 6.2]{MR1478817}.
    The assertion \ref{cor:relative_anabelian_conditionsigmalem1}, together with the snake lemma, implies that the left-hand and right-hand vertical arrows are bijective.
    Hence \ref{cor:relative_anabelian_conditionsigmalem3} follows.
    This completes the proof.
\end{proof}

Other generalizations of Theorem~\ref{thm:relative_anabelian} and Corollary~\ref{cor:relative_anabelian} can be considered under suitable modifications.
For instance, we can consider the $m$-step solvable generalization, see~\cite{yamaguchi2024refined} and \cite{MR4745885} for classical results on this topic.

\subsection{Counterexamples to the Grothendieck Conjecture when \texorpdfstring{no rational prime is invertible}{no rational prime is invertible}}\label{subsection1.2}

\subsubsection{}
In this subsection, we introduce some counterexamples to the Grothendieck conjecture when no rational prime is invertible on the base scheme.
We first recall a previous result for the \'etale fundamental groups of smooth curves.

\begin{theorem}[{\cite[Theorem C]{MR1305400}}]\label{thm:iharatriv}
    Let $K$ be a number field and let $\mathbb{P}^{1}_{K}$ be the projective $t$-line over $K$. 
    For $a\in K$, we denote by $D_a$ the Zariski closure in $\mathbb{P}^{1}_{\mathcal{O}_{K}}$ of the principal divisor defined by $t-a=0$, endowed with the reduced structure. Moreover, $D_\infty$ denotes the section at infinity, endowed with the reduced structure.
    Then $(\mathbb{P}^1_{\mathcal{O}_{K}},D_{0}\sqcup D_1\sqcup D_\infty)$ is a smooth curve over $\mathcal{O}_{K}$, and the morphism
    \begin{equation*}
        \Pi_{\mathbb{P}^1_{\mathcal{O}_{K}}\setminus D_{0}\sqcup D_1\sqcup D_\infty}
        \rightarrow
        \Pi_{\Spec(\mathcal{O}_{K})}
    \end{equation*}
    induced by the structure morphism is an isomorphism.
\end{theorem}

This theorem may be viewed as an arithmetic analogue of the Lefschetz theorem for projective surfaces over a field (see~\cite{MR1681810} for further details of this view).
As noted in \cite{MR1305400}, this theorem was first proved by T.~Saito, and Y.~Ihara provided an independent proof.

\subsubsection{}
Next, we generalize Theorem~\ref{thm:iharatriv}.
Recall that $\mathbb{L}(S)$ denotes the set of all rational primes that are invertible in the ring $\mathcal{O}_{K,S}$ of $S$-integers (see Notation~\ref{notation:OKS}).
    For any scheme $\mathcal{T}$, we denote by
    \begin{equation*}
        \mathrm{Pt}_{i}(\mathcal{T})
    \end{equation*}
    the set of all codimension $i$ points of $\mathcal{T}$.   
The following proposition was already proved by \cite[Proposition 9.3.6]{bost2022quasiprojectiveformalanalyticarithmeticsurfaces} when $S=\emptyset$:

\begin{definition}
    Let $(\mathcal{X}^{\mathrm{cpt}},\mathcal{D})$ be a smooth curve over $\mathcal{O}_{K,S}$, and set $\mathcal{X}\coloneqq \mathcal{X}^{\mathrm{cpt}}\setminus \mathcal{D}$ (see Definition~\ref{smcurvedef}).
    Let $\Sigma\subset \mathfrak{Primes}_\mathbb{Q}$. 
    Let $f:(\widetilde{\mathcal{X}}^{\Sigma})^{\mathrm{cpt}}\to \mathcal{X}^{\mathrm{cpt}}$ be a maximal connected pro-$\Sigma$ Galois covering of $\mathcal{X}^{\mathrm{cpt}}$ such that it is unramified outside $\mathcal{D}$.
    Write $\tilde{\mathcal{D}}^{\Sigma}$ for the inverse image of $\mathcal{D}$ by $f$.   
    Note that 
    \begin{enumerate}[(a)]
    \item $\tilde{\mathcal{X}}^{\Sigma}\coloneqq(\widetilde{\mathcal{X}}^{\Sigma})^{\mathrm{cpt}}\setminus\tilde{\mathcal{D}}^{\Sigma}$ is identified with the notation defined as in Notation~\ref{notation:etale};
    \item the pro-schems $(\widetilde{\mathcal{X}}^{\Sigma})^{\mathrm{cpt}}$, $\tilde{\mathcal{D}}^{\Sigma}$, and $\tilde{\mathcal{X}}^{\Sigma}$ are schemes by \cite[Tag 01YX]{stacks-project};
    \item when $\Sigma=\mathfrak{Primes}_\mathbb{Q}$, we write $\tilde{D}\coloneqq \tilde{D}^{\Sigma}$ for simplicity.
    \end{enumerate}
    Then, for any point $y\in \mathrm{Pt}_{0}(\tilde{\mathcal{D}}^{\Sigma})$ we set
    \begin{equation*}
        D_{y}=D_{y,\Pi_{\mathcal{X}}^{(\Sigma)}}\coloneqq \{\gamma\in \Pi_{\mathcal{X}}^{(\Sigma)}\mid \gamma(y)=y\},
    \end{equation*}
    \begin{equation*}
        I_{y}=I_{y,\Pi_{\mathcal{X}}^{(\Sigma)}}\coloneqq \{\gamma\in D_{y}\mid\gamma \text{ acts trivially on the residue field }\kappa(y)\},
    \end{equation*}
    and refer to these as the \textit{decomposition subgroup and inertia subgroup} at $y$, respectively.
\end{definition}

\begin{lemma}\label{lem:inertia_proL_triv}
    Let $K$ be a number field and let $S\subset\mathfrak{Primes}_{K}$.
    Let $(\mathcal{X}^{\mathrm{cpt}},\mathcal{D})$ be a smooth curve over $\mathcal{O}_{K,S}$, and
    set $\mathcal{X}\coloneqq \mathcal{X}^{\mathrm{cpt}}\setminus \mathcal{D}$.
    Then the inertia subgroup at every element in $\mathrm{Pt}_{0}(\tilde{\mathcal{D}})$ in $\Pi_{\mathcal{X}}$ is a pro-$\mathbb{L}(S)$ group.
\end{lemma}

\begin{proof}
    By definition, for every rational prime $p\in\mathfrak{Primes}_{\mathbb{Q}}\setminus \mathbb{L}(S)$, there exists a closed point of $\mathcal{D}$
    whose residue field is of characteristic $p$. 
    Let $\mathcal{Y} \to \mathcal{X}$ be a finite Galois covering with Galois group $G(\mathcal{Y}/\mathcal{X})$.
    (Note that $\mathcal{Y}\to\mathcal{X}$ is tamely ramified along $\mathcal{D}$ (see~\cite[Tag 0BSE]{stacks-project} for the definition), since $K$ is of characteristic $0$.)
    By Abhyankar's Lemma (see~\cite[Expos\'e XIII, Proposition 5.2]{MR0354651} or \cite[Tag 0EYG]{stacks-project}), the ramification indices of $\mathcal{Y}\to \mathcal{X}$ are coprime to $p$.
    Therefore, every prime factor of every inertia subgroup in $G(\mathcal{Y}/\mathcal{X})$ lies in $\mathbb{L}(S)$.
    By running over all such coverings $\mathcal{Y}\to\mathcal{X}$, the assertion follows.
\end{proof}

\begin{proposition}\label{compactification_noinvertibleprime}
    Let $K$ be a number field and let $S\subset\mathfrak{Primes}_{K}$.
    Let $(\mathcal{X}^{\mathrm{cpt}},\mathcal{D})$ be a smooth curve of type $(g,r)$ over $\mathcal{O}_{K,S}$, and
    set $\mathcal{X}\coloneqq \mathcal{X}^{\mathrm{cpt}}\setminus \mathcal{D}$.
    Assume that $\mathbb{L}(S)=\emptyset$.
    Then the natural morphism $\Pi_{\mathcal{X}}\to \Pi_{\mathcal{X}^{\mathrm{cpt}}}$ is an isomorphism.
    If, moreover, $g=0$, then the natural morphism $\Pi_{\mathcal{X}}\to G_{K,S}$ is an isomorphism.
\end{proposition}

\begin{proof}
    Let $\eta$ be the generic point of $\Spec(\mathcal{O}_{K,S})$. Set $X\coloneqq \mathcal{X}_{\eta}$ and $X^{\mathrm{cpt}}\coloneqq (\mathcal{X}^{\mathrm{cpt}})_{\eta}$.
    By the Zariski--Nagata purity theorem for the branch locus (see \cite[Expos\'e X, Th\'eor\`eme 3.1]{MR0354651} or \cite[Tag 0BMB]{stacks-project}), the kernel $\ker(\Delta_{X/K}\twoheadrightarrow\Delta_{X^{\mathrm{cpt}}/K})$ is generated by the inertia subgroups at the geometric points of $\mathcal{D}_{\eta}=X^{\mathrm{cpt}}\setminus X$ as profinite groups and is mapped onto $\ker(\Pi_{\mathcal{X}}\twoheadrightarrow\Pi_{\mathcal{X}^{\mathrm{cpt}}})$ by the natural morphism $\Delta_{X/K}\to \Pi_{\mathcal{X}}$.
    Therefore, we have the following exact sequence:
    \begin{equation*}
        1\to I
        \to \Pi_{\mathcal{X}}
        \to \Pi_{\mathcal{X}^{\mathrm{cpt}}}
        \to 1,
    \end{equation*}
    where $I$ is the closed subgroup of $\Pi_{\mathcal{X}}$ generated by the inertia subgroups as profinite groups.
    By Lemma~\ref{lem:inertia_proL_triv} and the hypothesis ``$\mathbb{L}(S)=\emptyset$", such inertia subgroups are trivial. Thus, $I$ is also trivial and therefore $\Pi_{\mathcal{X}}\twoheadrightarrow\Pi_{\mathcal{X}^{\mathrm{cpt}}}$ is an isomorphism.
    This completes the proof of the first assertion.
    In addition, assume $g=0$.
    By \cite[Expos\'e XIII, Lemme 4.2, Exemples 4.4]{MR0354651}, we have the following right-exact sequence:
    \begin{equation*}
        \Delta_{X^{\mathrm{cpt}}/K}
        \to \Pi_{\mathcal{X}^{\mathrm{cpt}}}\to G_{K,S}
        \to 1.
    \end{equation*}
    Since $\Delta_{X^{\mathrm{cpt}}/K}$ is trivial, the natural morphism $\Pi_{\mathcal{X}^{\mathrm{cpt}}}\twoheadrightarrow G_{K,S}$ is an isomorphism.
    This completes the proof of the second assertion.
\end{proof}

By using this proposition, we can find counterexamples as follows.

\begin{example}[Counterexamples to the Grothendieck conjecture over the rings of integers of number fields]\label{ex:arithmetic_surface_counterex}
    Let $K$ be a number field.
    Let us construct, for sufficiently large $K$, pairs of hyperbolic curves $\mathcal{X}_{1}$ and $\mathcal{X}_{2}$ over $\Spec(\mathcal{O}_K)$ such that $\mathcal{X}_{1}\not\simeq\mathcal{X}_{2}$ but $\Pi_{\mathcal{X}_{1}}\simeq \Pi_{\mathcal{X}_{2}}$.
    These pairs demonstrate the necessity of the existence of ``$p\in\Sigma$ that is invertible on $\mathcal{S}$'' in the hypothesis \ref{cor:relative_anabelian_conditionsigma0} of Corollary~\ref{cor:relative_anabelian_conditionsigma}.
    \begin{enumerate}[(1)]
        \item
              We first give examples of hyperbolic curves of genus $0$:
              As we have seen in Theorem~\ref{thm:iharatriv}, for any $K$, $(\mathbb{P}^1_{\mathcal{O}_{K}},D_{0}\sqcup D_1\sqcup D_\infty)$ is a smooth curve of type $(0,3)$ over $\mathcal{O}_{K}$ such that
              \begin{equation*}
                  \Pi_{\mathbb{P}^1_{\mathcal{O}_{K}}\setminus D_{0}\sqcup D_1\sqcup D_\infty}
                  \xrightarrow{\sim}
                  \Pi_{\Spec(\mathcal{O}_{K})}.
              \end{equation*}
              Let $p$ and $q$ be rational primes with $p<q$, and $\zeta_{pq}$ a primitive $pq$-th root of unity. Then, for any two distinct elements $a,b \in \{ 1,\zeta_{pq}, \ldots, \zeta_{pq}^{p-1}\}$, the difference $a-b$ is a unit in $\mathbb{Z}[\zeta_{pq}]$ (see \cite[Proposition~2.8]{MR1421575}).
              Hence, for any two distinct elements $a,b \in \{0, 1,\zeta_{pq}, \ldots, \zeta_{pq}^{p-1}, \infty\}$, the horizontal divisors $D_a$ and $D_b$ of $\mathbb{P}^1_{\mathbb{Z}[\zeta_{pq}]}$ do not meet.
              Therefore, the scheme $\mathbb{P}^{1}_{\mathbb{Z}[\zeta_{pq}]}\setminus \bigsqcup_{a} D_a$, where $a$ runs through $\{0, 1,\zeta_{pq}, \ldots, \zeta_{pq}^{p-1}, \infty\}$, is a smooth curve of type $(0,p+2)$ over $\mathbb{Z}[\zeta_{pq}]$
              (equivalently, the smooth curve $\mathbb{P}^{1}_{\mathbb{Q}(\zeta_{pq})}\setminus \bigsqcup_{a} D_a$ over $\mathbb{Q}(\zeta_{pq})$ has everywhere good reduction).
              Hence, by Proposition~\ref{compactification_noinvertibleprime}, we have an isomorphism
              \begin{equation*}
                  \Pi_{\mathbb{P}^{1}_{\mathbb{Z}[\zeta_{pq}]}\setminus \bigsqcup_{a} D_a}
                  \xrightarrow{\sim}
                  \Pi_{\Spec(\mathbb{Z}[\zeta_{pq}])}.
              \end{equation*}
              For instance, setting
              \begin{equation*}
                  \mathcal{X}_{1}=\mathbb{P}^{1}_{\mathbb{Z}[\zeta_{6}]}\setminus D_{0}\sqcup D_1\sqcup D_\infty,
                  \,\mathcal{X}_{2}=\mathbb{P}^{1}_{\mathbb{Z}[\zeta_{6}]}\setminus D_{0}\sqcup D_1\sqcup D_{\zeta_{6}}\sqcup D_\infty
              \end{equation*}
              yields a desired pair as above.
        \item
              It is known that there exists a number field $K$ and a smooth proper curve $X^{\mathrm{cpt}}/K$ of genus $g\geq 2$ with everywhere good reduction (for example, see \cite[Proposition 0.2]{MR3065454} in the case where $g=2$).
              Let $\mathcal{X}^{\mathrm{cpt}}/\mathcal{O}_K$ be a smooth proper model of the above curve $X^{\mathrm{cpt}}/K$. After replacing $K$ by a finite extension if necessary, we may assume that $X^{\mathrm{cpt}}(K)\neq\varnothing$. By properness, any $K$-rational point of $X^{\mathrm{cpt}}$ extends uniquely to a section $s:\Spec(\mathcal{O}_K)\to\mathcal{X}^{\mathrm{cpt}}$.
              Let $\mathcal{D}\coloneqq s(\Spec(\mathcal{O}_K))\subset\mathcal{X}^{\mathrm{cpt}}$ be the corresponding relative effective Cartier divisor.
              Then $(\mathcal{X}^{\mathrm{cpt}},\mathcal{D})$ is a smooth curve of type $(g,1)$ over $\mathcal{O}_K$.
              By Proposition~\ref{compactification_noinvertibleprime}, the canonical morphism
              \begin{equation*}
              \Pi_{\mathcal{X}^{\mathrm{cpt}}\setminus\mathcal{D}} \longrightarrow \Pi_{\mathcal{X}^{\mathrm{cpt}}} 
              \end{equation*} 
              is an isomorphism.
              Thus, setting $\mathcal{X}_{1}=\mathcal{X}^{\mathrm{cpt}}$ and $\mathcal{X}_{2}=\mathcal{X}^{\mathrm{cpt}}\setminus\mathcal{D}$
              yields a desired pair.
              Similarly, to construct a desired pair in the case where $g=1$, it suffices to find a smooth curve $(\mathcal{X}^{\mathrm{cpt}},\mathcal{D})$ of type $(1,2)$ over $\mathcal{O}_K$.
    \end{enumerate}
\end{example}

\begin{remark}[Finiteness of horizontal divisors on hyperbolic curves over $\mathcal{O}_{K,S}$]
    Let $K$ be a number field and let $S\subset\mathfrak{Primes}_{K}$.
    Let $\mathcal{X}$ be a hyperbolic curve over $\mathcal{O}_{K,S}$.
    Then Siegel's theorem and the Mordell conjecture imply that the set $\mathcal{X}(\mathcal{O}_{K,S})$ of $\mathcal{O}_{K,S}$-rational points of $\mathcal{X}$ is finite (see~\cite[Remark D.9.2.2]{MR1745599}).
    Therefore, if $G_{K,S}$ is a finite group (e.g., $K=\mathbb{Q}$ and $S=\emptyset$), then there exist finitely many horizontal divisors of $\mathcal{X}$ that are \'etale over $\mathcal{O}_{K,S}$.
    In particular, there exist finitely many smooth curves over $\mathcal{O}_{K,S}$ obtained from $\mathcal{X}$ by removing a union of horizontal divisors.
    Therefore, to construct pairs as in Example~\ref{ex:arithmetic_surface_counterex}, we have to work within this constraint.
    Moreover, when $K=\mathbb{Q}$ and $S=\emptyset$, every hyperbolic curve over $\mathbb{Z}$ is isomorphic to $\mathbb{P}^{1}_{\mathbb{Z}}\setminus D_{0}\sqcup D_{1}\sqcup D_{\infty}$ (see \cite[Proposition A2]{MR2123370}), and hence we cannot construct a pair over $\mathbb{Z}$.
\end{remark}

\subsubsection{}
At the end of this subsection, we introduce other counterexamples to the Grothendieck conjecture as follows:

\begin{example}[Counterexamples to the Grothendieck conjecture over the rings of integers of $p$-adic fields]\label{ex:p-adic_counterex}
    Let $K$ be a $p$-adic field, $\mathcal{O}_K$ its ring of integers, and $k$ its residue field.
    Note that every rational prime different from $p$ is invertible on $\Spec(\mathcal{O}_K)$, whereas $p$ is not.
    Then for any set $\Sigma$ of rational primes, the scheme $\Spec(\mathcal{O}_K)$ cannot satisfy the condition \ref{cor:relative_anabelian_conditionsigma0} in Corollary~\ref{cor:relative_anabelian_conditionsigma}, since $K$ is sub-$p$-adic only for $p$.
    Moreover, it is known that, for any proper smooth curve $X$ of genus $g\geq 2$ over $k$, there exist two non-isomorphic liftings of $X$ to proper smooth curves $\mathcal{X}_{1}$ and $\mathcal{X}_{2}$ over $\Spec(\mathcal{O}_K)$.
    Indeed, this can be seen by following the proof of \cite[Theorem 22.1]{MR2583634}, together with the fact that $\dim_k H^1(X, T_X) = 3g - 3 > 0$, where $T_X$ denotes the tangent sheaf of $X/k$.
    By \cite[Tag 0A48]{stacks-project}, $\Pi_{\mathcal{X}_{1}}$ and $\Pi_{\mathcal{X}_{2}}$ are isomorphic to $\Pi_X$, and hence these examples demonstrate the failure of the Grothendieck conjecture over $\Spec(\mathcal{O}_K)$.
\end{example}

\section{Semi-absolute version of the Grothendieck conjecture over rings of \texorpdfstring{$S$}{S}-integers}

In this section, we investigate the semi-absolute version of the Grothendieck conjecture for hyperbolic curves over $\mathcal{O}_{K,S}$, where $K$ is a number field and $S\subset\mathfrak{Primes}_{K}$.
Let $\mathbb{L}$ be a set of rational primes that are invertible in $\mathcal{O}_{K,S}$, and fix a rational prime $p\in\mathbb{L}$.
Let $(\mathcal{X}^{\mathrm{cpt}},\mathcal{D})$ be a smooth curve of type $(g,r)$ over $\mathcal{O}_{K,S}$, and
set $\mathcal{X}\coloneqq \mathcal{X}^{\mathrm{cpt}}\setminus \mathcal{D}$.
Let $\eta$ be the generic point of $\Spec(\mathcal{O}_{K,S})$, and set $X\coloneqq \mathcal{X}_{\eta}$.
In the first subsection, we recall the separatedness property for inertia subgroups and then reconstruct the cusps, the pair $(g,r)$, and the cyclotomic character by applying Mochizuki's results.
In the last two subsections, we prove Theorems A and B in the Introduction.

\subsection{Separatedness property for inertia subgroups and Galois-preserving isomorphisms}\label{subsection:Noncommutative_weight_sep_inertia}

\subsubsection{}
In this subsection, we introduce several basic properties that are used repeatedly in the rest of this section.
First, we introduce the separatedness property for inertia subgroups of $\Delta_{\mathcal{X}/\mathcal{O}_{K,S}}^{\mathbb{L}}$.
This separatedness property is classical; see, for instance, \cite[Section~2]{MR1072981} and \cite[Lemma~1.3.7]{MR2059759}.

\begin{lemma}\label{inerrev}
    Let $K$ be a number field, $S\subset\mathfrak{Primes}_{K}$, and let $\mathbb{L}$ be a nonempty set of rational primes that are invertible in $\mathcal{O}_{K,S}$.
    Let $(\mathcal{X}^{\mathrm{cpt}},\mathcal{D})$ be a hyperbolic curve of type $(g,r)$ over $\mathcal{O}_{K,S}$, and set $\mathcal{X}\coloneqq \mathcal{X}^{\mathrm{cpt}}\setminus\mathcal{D}$.
    Then for any $\tilde{\zeta},\tilde{\zeta}'\in\mathrm{Pt}_{0}(\tilde{\mathcal{D}}^{\mathbb{L}})$, the following are equivalent:
    \begin{enumerate}[(a)]
        \item\label{inerrev1}
              $\tilde{\zeta}=\tilde{\zeta}'$.
        \item \label{inerrev2}
              $I_{\tilde{\zeta},\Pi_{\mathcal{X}}^{(\mathbb{L})}}=I_{\tilde{\zeta}',\Pi_{\mathcal{X}}^{(\mathbb{L})}}$.
        \item\label{inerrev3}
              $I_{\tilde{\zeta},\Pi_{\mathcal{X}}^{(\mathbb{L})}}$ and $I_{\tilde{\zeta}',\Pi_{\mathcal{X}}^{(\mathbb{L})}}$ are commensurable (i.e., their intersection is open in both of them).
    \end{enumerate}
    In particular, the normalizer of $I_{\tilde{\zeta},\Pi_{\mathcal{X}}^{(\mathbb{L})}}$ in $\Pi_{\mathcal{X}}^{(\mathbb{L})}$ coincides with
    $D_{\tilde{\zeta},\Pi_{\mathcal{X}}^{(\mathbb{L})}}$.
\end{lemma}

\begin{proof}
    Let $\eta$ be the generic point of $\Spec(\mathcal{O}_{K,S})$, set $X\coloneqq\mathcal{X}_{\eta}$, and set $X^{\mathrm{cpt}}\coloneqq(\mathcal X^{\mathrm{cpt}})_{\eta}$.
    By the natural morphism
    \[
        \Delta_{X/K}^{\mathbb L}
        \xrightarrow{\sim}
        \Delta_{\mathcal X/\mathcal O_{K,S}}^{\mathbb L}.
    \]
    in Proposition~\ref{lem:stix_exact}, we identify $\Delta_{\mathcal X/\mathcal O_{K,S}}^{\mathbb L}$ with $\Delta_{X/K}^{\mathbb L}$.
    Under this identification, the elements of $\mathrm{Pt}_{0}(\tilde{\mathcal D}^{\mathbb L})$ are identified with the open edges of the semi-graph of anabelioids of pro-$\mathbb L$ PSC-type associated to the pointed smooth curve $(X^{\mathrm{cpt}},\mathcal D_{\eta})_{\overline K}$.
    Thus \cite[Proposition~1.2(i)]{MR2365351} applies to the corresponding cuspidal edge-like subgroups of the geometric fundamental group.
    It implies that two such inertia subgroups are commensurable if and only if they arise from the same cusp.
    This proves the equivalence of \ref{inerrev1}, \ref{inerrev2}, and \ref{inerrev3}.
    The last assertion follows from these equivalences.
\end{proof}

\subsubsection{}
Next, we introduce a property for isomorphisms of \'etale fundamental groups called \textit{Galois-preserving}, which is often used in anabelian geometry (e.g., \cite[Definition~3.1(ii)]{MR3676682}), and show a basic property of Galois-preserving isomorphisms.

\begin{definition}\label{def:GP}
    Let $i$ range over $\{1,2\}$.
    Let $K_i$ be a number field, $S_i\subset \mathfrak{Primes}_{K_i}$, and let $\mathbb{L}_i$ be a nonempty set of rational primes invertible in $\mathcal{O}_{K_i,S_i}$.
    Let $\mathcal{X}_i$ be a smooth curve over $\mathcal{O}_{K_i,S_i}$.
    Consider the homotopy exact sequence
    \begin{equation*}
        1\longrightarrow \Delta_{\mathcal{X}_i/\mathcal{O}_{K_{i},S_i}}^{\mathbb{L}_i}
        \longrightarrow \Pi_{\mathcal{X}_i}^{(\mathbb{L}_i)}
        \longrightarrow G_{K_i,S_i}
        \longrightarrow 1.
    \end{equation*}
    We say that an isomorphism $\Phi:\Pi_{\mathcal{X}_1}^{(\mathbb{L}_1)}\xrightarrow{\sim}\Pi_{\mathcal{X}_2}^{(\mathbb{L}_2)}$ is \textit{Galois-preserving} if
    \[
        \Phi\bigl(\Delta_{\mathcal{X}_1/\mathcal{O}_{K_{1},S_1}}^{\mathbb{L}_1}\bigr)
        =
        \Delta_{\mathcal{X}_2/\mathcal{O}_{K_{2},S_2}}^{\mathbb{L}_2}.
    \]
    In this case, $\Phi$ induces an isomorphism $\Phi_G:G_{K_1,S_1}\xrightarrow{\sim}G_{K_2,S_2}$.
    We denote by 
    \[
    \Isom^{\mathrm{GP}}\bigl(\Pi^{(\mathbb{L}_{1})}_{\mathcal{X}_1}, \Pi^{(\mathbb{L}_{2})}_{\mathcal{X}_2}\bigr)
    \] the set of all Galois-preserving isomorphisms, and set
    \begin{equation*}
        \Isom^{\mathrm{GP},\Out}\bigl(\Pi^{(\mathbb{L}_{1})}_{\mathcal{X}_1}, \Pi^{(\mathbb{L}_{2})}_{\mathcal{X}_2}\bigr)
        \coloneqq
        \Isom^{\mathrm{GP}}\bigl(\Pi^{(\mathbb{L}_{1})}_{\mathcal{X}_1}, \Pi^{(\mathbb{L}_{2})}_{\mathcal{X}_2}\bigr)/\Inn(\Delta_{\mathcal{X}_2/\mathcal{O}_{K_{2},S_2}}^{\mathbb{L}_2}),
    \end{equation*}
    where $\Inn(\Delta_{\mathcal{X}_2/\mathcal{O}_{K_{2},S_2}}^{\mathbb{L}_2})$ denotes the subgroup of
    $\Isom_{G_{K_{2},S_{2}}}\bigl(\Pi^{(\mathbb{L}_{2})}_{\mathcal{X}_2}, \Pi^{(\mathbb{L}_{2})}_{\mathcal{X}_2}\bigr)$ consisting of inner automorphisms given by conjugation by elements of $\Delta_{\mathcal{X}_2/\mathcal{O}_{K_{2},S_2}}^{\mathbb{L}_2}$.
\end{definition}

\begin{lemma}\label{gallem}
    We keep the notation of Definition~\ref{def:GP}.
    Let $\Phi:\Pi_{\mathcal{X}_1}^{(\mathbb{L}_1)}\xrightarrow{\sim}\Pi_{\mathcal{X}_2}^{(\mathbb{L}_2)}$ be a Galois-preserving isomorphism.
    Then the following hold:
    \begin{enumerate}[(1)]
        \item\label{gallem1}
              For each $i\in\{1,2\}$, the following are equivalent:
              \begin{enumerate}[(a)]
                  \item\label{gallem1-1}
                        $\mathcal{X}_i$ is hyperbolic.
                  \item\label{gallem1-2}
                        $\Delta_{\mathcal{X}_i/\mathcal{O}_{K_i,S_i}}^{\mathbb{L}_i}$ is not abelian.
              \end{enumerate}
              In particular, $\mathcal{X}_1$ is hyperbolic if and only if $\mathcal{X}_2$ is hyperbolic.
        \item\label{gallem2}
              If at least one of $\Delta_{\mathcal{X}_1/\mathcal{O}_{K_1,S_1}}^{\mathbb{L}_1}$ and $\Delta_{\mathcal{X}_2/\mathcal{O}_{K_2,S_2}}^{\mathbb{L}_2}$ is not abelian, then $\mathbb{L}_{1}=\mathbb{L}_{2}$.
    \end{enumerate}
\end{lemma}

\begin{proof}
    Let $\eta_{i}$ be the generic point of $\Spec(\mathcal{O}_{K_{i},S_{i}})$.
    For simplicity, set $\mathcal{O}_{i}\coloneqq \mathcal{O}_{K_i,S_i}$ and $X_{i}\coloneqq \mathcal{X}_{i,\eta_{i}}$.

    \noindent
    \ref{gallem1}
    By the definition, $\mathcal{X}_{i}$ is hyperbolic if and only if $X_{i}$ is hyperbolic.
    Hence, by Lemma~\ref{basicconditioneq}, $\mathcal{X}_{i}$ is hyperbolic if and only if $\Delta_{X_{i}/K_{i}}^{\mathbb{L}_{i}}$ is not abelian.
    If $\mathcal{X}_{i}$ is hyperbolic, then Proposition~\ref{lem:stix_exact} implies that $\Delta_{X_{i}/K_{i}}^{\mathbb{L}_{i}}\xrightarrow{\sim}\Delta_{\mathcal{X}_{i}/\mathcal{O}_{i}}^{\mathbb{L}_{i}}$, and hence $\Delta_{\mathcal{X}_{i}/\mathcal{O}_{i}}^{\mathbb{L}_{i}}$ is also not abelian.
    If $\Delta_{\mathcal{X}_i/\mathcal{O}_{i}}^{\mathbb{L}_i}$ is not abelian, then, by the surjectivity of the natural morphism $\Delta_{X_{i}/K_{i}}^{\mathbb{L}_{i}}\twoheadrightarrow \Delta_{\mathcal{X}_{i}/\mathcal{O}_{i}}^{\mathbb{L}_{i}}$, so is $\Delta_{X_i/K_i}^{\mathbb{L}_i}$.
    Therefore, $\mathcal{X}_i$ is hyperbolic.
    The second assertion follows immediately from the first assertion and the Galois-preserving property of $\Phi$.

    \noindent
    \ref{gallem2}
    Since $\Phi$ is Galois-preserving, $\Delta_{\mathcal{X}_1/\mathcal{O}_{1}}^{\mathbb{L}_1}$ and $\Delta_{\mathcal{X}_2/\mathcal{O}_{2}}^{\mathbb{L}_2}$ are isomorphic, and hence both of them are not abelian.
    In particular, by Lemma~\ref{basicconditioneq}, we have that $\Delta_{\mathcal{X}_{i}/\mathcal{O}_{i}}^{\mathrm{ab},p_{i}}$ is nontrivial for $p_{i}\in\mathbb{L}_{i}$.
    Hence the assertion follows.
\end{proof}

\begin{proposition}\label{cor:recon_sev}
    We keep the notation of Definition~\ref{def:GP}.
    Let $\Phi:\Pi_{\mathcal{X}_1}^{(\mathbb{L}_1)}\xrightarrow{\sim}\Pi_{\mathcal{X}_2}^{(\mathbb{L}_2)}$ be a Galois-preserving isomorphism and let $\Phi_G:G_{K_1,S_1}\xrightarrow{\sim}G_{K_2,S_2}$ be the isomorphism induced by $\Phi$.
    Assume the following condition:
    \begin{enumerate}[(a)]
        \item\label{cor:recon_sevas2}
              At least one of $\Delta_{\mathcal{X}_{1}/\mathcal{O}_{K_{1},S_{1}}}^{\mathbb{L}_{1}}$ and $\Delta_{\mathcal{X}_{2}/\mathcal{O}_{K_{2},S_{2}}}^{\mathbb{L}_{2}}$ is not abelian.
    \end{enumerate}
    Then the following hold:
    \begin{enumerate}[(1)]
        \item\label{cor:recon_sev3}
              The isomorphism $\Phi$ induces a bijection
              \begin{equation}\label{hthgargrwgerag}
                  \mathrm{Pt}_{0}(\tilde{\mathcal{D}}_{1}^{\mathbb{L}_{1}})
                  \xrightarrow{\sim}
                  \mathrm{Pt}_{0}(\tilde{\mathcal{D}}_{2}^{\mathbb{L}_{2}})
              \end{equation}
              that is compatible with the actions of $\Delta_{\mathcal{X}_{1}/\mathcal{O}_{K_1,S_1}}^{\mathbb{L}_{1}}$ and $\Delta_{\mathcal{X}_{2}/\mathcal{O}_{K_2,S_2}}^{\mathbb{L}_{2}}$ relative to $\Phi$.
              In particular, $\Phi$ preserves inertia subgroups and cuspidal decomposition subgroups of $\Pi_{\mathcal{X}_{i}}^{(\mathbb{L}_{i})}$.
        \item \label{recogreq}
        Let $\mathcal{X}_{i}$ be of type $(g_i,r_i)$. 
        Then  $g_{1}=g_{2}$ and $r_{1}=r_{2}$ hold.
        \item\label{cor:recon_sev4}
              Let $\chi^{\mathbb{L}_{i}\text{-}\mathrm{cycl}}_{i}:G_{K_{i},S_{i}}\to (\hat{\mathbb{Z}}^{\mathbb{L}_i})^{\times}$ be the $\mathbb{L}_{i}$-adic cyclotomic character of $G_{K_{i},S_{i}}$.
              Then
              \[
                  \chi^{\mathbb{L}_{2}\text{-}\mathrm{cycl}}_{2}\circ \Phi_{G}=\chi^{\mathbb{L}_{1}\text{-}\mathrm{cycl}}_{1}
              \]
              holds.
    \end{enumerate}
\end{proposition}

\begin{proof}
    The hypothesis \ref{cor:recon_sevas2} and Lemma~\ref{gallem}\ref{gallem1} imply that both $\mathcal{X}_{1}$ and $\mathcal{X}_{2}$ are hyperbolic.
    Moreover, Lemma~\ref{gallem}\ref{gallem2} implies $\mathbb{L}_{1}=\mathbb{L}_{2}$.
    We write this set as $\mathbb{L}$ for simplicity.

    \noindent
    \ref{cor:recon_sev3}
    For each $i$, let $\eta_i$ be the generic point of $\Spec(\mathcal O_{K_i,S_i})$, set $X_i\coloneqq\mathcal X_{i,\eta_i}$, and set $X_i^{\mathrm{cpt}}\coloneqq(\mathcal X_i^{\mathrm{cpt}})_{\eta_i}$.
    Via Proposition~\ref{lem:stix_exact}, the group $\Delta_{\mathcal{X}_{i}/\mathcal{O}_{K_i,S_i}}^{\mathbb L}$ is identified with the geometrically pro-$\mathbb L$ fundamental group of $X_i$.
    Equivalently, it is the fundamental group of the semi-graph of anabelioids of pro-$\mathbb L$ PSC-type associated to the pointed smooth curve $(X_i^{\mathrm{cpt}},\mathcal D_{i,\eta_i})_{\overline K_i}$.
    Choose a rational prime $p\in\mathbb L$.
    Since $p$ is invertible in $\mathcal O_{K_i,S_i}$, every prime of $K_i$ above $p$ belongs to $S_i$.
    Hence the usual $p$-adic cyclotomic character factors through $G_{K_i,S_i}$ and its image is open in $\mathbb Z_{p}^{\times}$, i.e.,  the two outer actions are $p$-cyclotomically full in the sense of \cite[Definition~2.3(ii)]{MR2365351}.
    Applying \cite[Corollary~2.7(i)]{MR2365351}, we obtain that the restriction of $\Phi$ on the geometric part is group-theoretically cuspidal.
    Thus, $\Phi$ preserves inertia subgroups.
    By Lemma~\ref{inerrev}, this gives the desired bijection \eqref{hthgargrwgerag}.
    Moreover, the corresponding decomposition subgroups are precisely the normalizers of these inertia subgroups by Lemma~\ref{inerrev}.
    Hence $\Phi$ also preserves cuspidal decomposition subgroups.
    This completes the proof.

    \noindent
    \ref{recogreq}
    The bijection in \ref{cor:recon_sev3} implies $r_{1}=r_{2}$.
    Choose a rational prime $p\in\mathbb{L}$.
    The isomorphism $\Phi$ induces an isomorphism between $\Delta_{\mathcal{X}_{1}/\mathcal{O}_{K_1,S_1}}^{\mathrm{ab},p}$ and $\Delta_{\mathcal{X}_{2}/\mathcal{O}_{K_2,S_2}}^{\mathrm{ab},p}$.
    Hence we obtain  
    \[
        \dim_{\mathbb{Q}_{p}}\bigl(\Delta_{\mathcal{X}_{1}/\mathcal{O}_{K_1,S_1}}^{\mathrm{ab},p}\otimes_{\mathbb{Z}_{p}}\mathbb{Q}_{p}\bigr)
        =
        \dim_{\mathbb{Q}_{p}}\bigl(\Delta_{\mathcal{X}_{2}/\mathcal{O}_{K_2,S_2}}^{\mathrm{ab},p}\otimes_{\mathbb{Z}_{p}}\mathbb{Q}_{p}\bigr).
    \]
    Since this common dimension is $2g_i+r_i-\varepsilon_i$, where $\varepsilon_i=1$ if $r_i\geq 1$ and $\varepsilon_i=0$ if $r_i=0$, the equality $r_{1}=r_{2}$ implies $g_{1}=g_{2}$.

    \noindent
    \ref{cor:recon_sev4}
    The assertion follows from \ref{cor:recon_sev3} and \cite[Lemma~2.1]{MR2365351}.
\end{proof}

\subsection{Proof of the semi-absolute Grothendieck conjecture over rings of \texorpdfstring{$S$}{S}-integers with sufficiently large upper Dirichlet densities}\label{subsection:proof_of_shimizu_type_result}

\subsubsection{}
By the Neukirch--Uchida theorem, the semi-absolute and relative versions of the Grothendieck conjecture over spectra of number fields are equivalent.
In this section, from Theorem~\ref{thm:relative_anabelian} and the Neukirch--Uchida type result {\cite[Theorem 3.4 and Remark 4.7]{MR4626871}} for $\mathcal{O}_{K,S}$, we obtain an affirmative result for the semi-absolute version of the Grothendieck conjecture over $\mathcal{O}_{K,S}$ when $S$ is sufficiently large (see~Theorem~\ref{thm:semi-absolute-main} below).
First, we recall the following theorem:

\begin{theorem}[{\cite[Theorem 3.4 and Remark 4.7]{MR4626871}}]\label{Thm:shimizurestriction}
    Let $i$ range over $\{1,2\}$.
    Let $K_{i}$ be a number field, and let $S_{i}\subset \mathfrak{Primes}_{K_{i}}$.
    Let $\tilde{\mathcal{O}}_{K_{i},S_{i}}$ be the normalization of $\mathcal{O}_{K_i,S_i}$ in a maximal Galois extension of $K_i$ unramified outside $S_i$.
    Assume the following conditions:
    \begin{enumerate}[(a)]
        \item\label{Thm:shimizurestriction1}
              $\#\mathbb{L}(S_{1})\geq 2$ and $\#\mathbb{L}(S_{2})\geq 2$.
        \item\label{Thm:shimizurestriction2}
              For at least one index $j\in\{1,2\}$ and every finite Galois extension $L_j/K_j$ unramified outside $S_{j}$,
              \begin{equation*}
                  \delta_{\mathrm{sup}}\bigl(\mathbb{L}(S_{j})\cap \mathrm{cs}(L_j/\mathbb{Q})\bigr)>0
              \end{equation*}
              holds, where $\mathrm{cs}(L_j/\mathbb{Q})$ denotes the set of all rational primes that split completely in $L_j/\mathbb{Q}$ and $\delta_{\mathrm{sup}}(\ast)$ denotes the upper Dirichlet density of $\ast$, see~\cite[Notations]{MR4402495}.
        \item\label{Thm:shimizurestriction3}
              Set $k\in\{1,2\}\setminus\{j\}$, where $j$ denotes the index in \ref{Thm:shimizurestriction2}.
              Then there exists a rational prime
              $\ell\in\mathbb{L}(S_{1})\cap\mathbb{L}(S_{2})$ such that $S_{k}$ satisfies the condition $(\star_\ell)$ in \cite[Definition~1.16]{MR4402495}.
    \end{enumerate}
    Then the canonical map
    \begin{equation*}
        \Isom\bigl(\tilde{\mathcal{O}}_{K_{2},{S_2}}/\mathcal{O}_{{K_2},S_2},\,\tilde{\mathcal{O}}_{K_{1}, S_1}/\mathcal{O}_{{K_1},S_1}\bigr)
        \longrightarrow
        \Isom\bigl(G_{K_1,S_1},\,G_{K_2,S_2}\bigr)
    \end{equation*}
    is bijective.
\end{theorem}

\begin{remark}[A sufficient condition for the hypotheses \ref{Thm:shimizurestriction1}, \ref{Thm:shimizurestriction2}, and \ref{Thm:shimizurestriction3} of Theorem~\ref{Thm:shimizurestriction}]\label{rem:shimizu}
    Note that the hypothesis \ref{Thm:shimizurestriction2} of Theorem~\ref{Thm:shimizurestriction} holds if $\delta_{\mathrm{sup}}\bigl(\mathbb{L}(S_{j})\bigr)=1$.
    Indeed, in this case, by \cite[Lem.~4.6]{MR4402495} and Chebotarev's density theorem, we have
    \begin{equation*}
        \delta_{\mathrm{inf}}\bigl(\mathfrak{Primes}_\mathbb{Q}\setminus \mathbb{L}(S_{j}) \bigr)
        = \delta\bigl(\mathfrak{Primes}_\mathbb{Q} \bigr)
        - \delta_{\mathrm{sup}}\bigl(\mathbb{L}(S_{j})\bigr)
        = 1-1=0,
    \end{equation*}
    and hence
    \begin{eqnarray*}
        \delta_{\mathrm{sup}}\bigl(\mathbb{L}(S_{j})\cap \mathrm{cs}(L_j/\mathbb{Q})\bigr)
        &\geq&
        \delta\bigl(\mathrm{cs}(L_j/\mathbb{Q})\bigr)
        - \delta_{\mathrm{inf}}\bigl((\mathfrak{Primes}_\mathbb{Q}\setminus \mathbb{L}(S_{j})) \cap \mathrm{cs}(L_j/\mathbb{Q}) \bigr)\\
        &=& 1/[L_j:\mathbb{Q}] - 0 > 0.
    \end{eqnarray*}
    Moreover, about the hypothesis \ref{Thm:shimizurestriction3}, the set $S_{k}$ satisfies $(\star_\ell)$ whenever $\delta_{\mathrm{sup}}\bigl(S_{k}\bigr)>0$ by \cite[Prop.~1.20]{MR4402495}.
    Therefore, the condition $\delta_{\mathrm{sup}}(\mathbb{L}(S_1))=\delta_{\mathrm{sup}}(\mathbb{L}(S_2))=1$ is sufficient for the hypotheses \ref{Thm:shimizurestriction1}, \ref{Thm:shimizurestriction2}, and \ref{Thm:shimizurestriction3} of Theorem~\ref{Thm:shimizurestriction} to hold.
\end{remark}

The following is the second main theorem of the present paper:

\begin{theorem}\label{thm:semi-absolute-main}
    Let $i$ range over $\{1,2\}$.
    Let $K_i$ be a number field, $S_i\subset \mathfrak{Primes}_{K_i}$, and let $\mathbb{L}_i$ be a nonempty set of rational primes that are invertible in $\mathcal{O}_{K_i,S_i}$.
    Let $\mathcal{X}_i$ be a smooth curve over $\mathcal{O}_{K_i,S_i}$.
    Assume the following conditions:
    \begin{enumerate}[(a)]
        \item\label{as:thm:semi-absolute-main2}
              At least one of $\mathcal{X}_1$ and $\mathcal{X}_2$ is hyperbolic.
        \item\label{as:thm:semi-absolute-main3}
              The data $(K_1,S_1,\mathbb{L}(S_1))$ and $(K_2,S_2,\mathbb{L}(S_2))$ satisfy the hypotheses \ref{Thm:shimizurestriction1}, \ref{Thm:shimizurestriction2}, and \ref{Thm:shimizurestriction3} of Theorem~\ref{Thm:shimizurestriction}.
    \end{enumerate}
    Then the natural map
    \begin{equation*}
        \Isom\bigl(\mathcal{X}_1/\mathcal{O}_{K_1,S_1},\,\mathcal{X}_2/\mathcal{O}_{K_2,S_2}\bigr)
        \longrightarrow
        \Isom^{\mathrm{GP},\Out}\bigl(\Pi^{(\mathbb{L}_{1})}_{\mathcal{X}_1},\,\Pi^{(\mathbb{L}_{2})}_{\mathcal{X}_2}\bigr)
    \end{equation*}
    is bijective, where the right-hand side is the set defined as in Definition~\ref{def:GP}.
\end{theorem}

\begin{proof}
    Take $\Phi\in \Isom^{\mathrm{GP}}\bigl(\Pi^{(\mathbb{L}_{1})}_{\mathcal{X}_1},\Pi^{(\mathbb{L}_{2})}_{\mathcal{X}_2}\bigr)$, and set $\Phi_G:\,G_{K_1,S_1}\xrightarrow{\sim}G_{K_2,S_2}$ for the isomorphism induced by $\Phi$.
    For simplicity, set $\mathcal{O}_{i}\coloneqq \mathcal{O}_{K_{i},S_i}$.
    The hypothesis \ref{as:thm:semi-absolute-main2} and Lemma~\ref{gallem}\ref{gallem1} imply that both $\mathcal{X}_{1}$ and $\mathcal{X}_{2}$ are hyperbolic.
    Moreover, Lemma~\ref{gallem}\ref{gallem2} implies that $\mathbb{L}_{1}=\mathbb{L}_{2}$.
    For simplicity, we write $\mathbb{L}$ instead of $\mathbb{L}_{1}(=\mathbb{L}_{2})$.

    By hypothesis \ref{as:thm:semi-absolute-main3} and Theorem~\ref{Thm:shimizurestriction}, there exists a unique isomorphism of arithmetic schemes $\Spec(\mathcal O_1)\xrightarrow{\sim}\Spec(\mathcal O_2)$, or equivalently a unique ring isomorphism $\phi_{o}:\mathcal{O}_{2}\xrightarrow{\sim}\mathcal{O}_{1}$, such that $\Phi_{G}$ comes from $\phi_{o}$ up to unique inner automorphism of $G_{K_{2},S_{2}}$.
    Let $\mathcal{X}_{1,\mathcal{O}_{2}}\to\Spec(\mathcal{O}_{2})$ be the base change of $\mathcal{X}_{1}$ via $\Spec(\phi_{o}^{-1}):\Spec(\mathcal{O}_{2})\xrightarrow{\sim}\Spec(\mathcal{O}_{1})$, i.e., we take it via the fiber product diagram
    \begin{equation*}
        \xymatrix@C=40pt{
        \mathcal{X}_{1,\mathcal{O}_{2}}\ar@{.>}[d]\ar@{.>}[r]^{\phi^{\ast}}& \mathcal{X}_{1}\ar[d]\\
        \Spec(\mathcal{O}_{2})\ar[r]^{\Spec(\phi_{o}^{-1})}&\Spec(\mathcal{O}_{1}).
        }
    \end{equation*}
    Then $\Phi$ induces a $G_{K_{2},S_{2}}$-isomorphism $\Phi^{\ast}:\Pi^{(\mathbb{L})}_{\mathcal{X}_{1,\mathcal{O}_{2}}}\xrightarrow{\sim}\Pi^{(\mathbb{L})}_{\mathcal{X}_{2}}$ as the composite

    \begin{equation*}
        \xymatrix@C=40pt{
        \Pi_{\mathcal{X}_{1,\mathcal{O}_{2}}}^{(\mathbb{L})}\ar[r]^-{\Pi(\phi^{\ast})} &\Pi_{\mathcal{X}_{1}}^{(\mathbb{L})} \ar[r]^-{\Phi}& \Pi_{\mathcal{X}_{2}}^{(\mathbb{L})}.
        }
    \end{equation*}
    Hence, by Theorem~\ref{thm:relative_anabelian} and the hypothesis \ref{as:thm:semi-absolute-main2}, there exists a unique isomorphism $\phi_{1}\in\Isom_{\mathcal{O}_{2}}\bigl(\mathcal{X}_{1,\mathcal{O}_{2}}, \mathcal{X}_2\bigr)$ such that $\Phi^{\ast}$ comes from $\phi_{1}$ up to unique inner automorphism of $\Delta_{\mathcal{X}_2/\mathcal{O}_{2}}^{\mathbb{L}}$.
    Therefore, $\Phi$ comes from $\phi_{1}\circ(\phi^{\ast})^{-1}$ up to unique inner automorphism of $\Delta_{\mathcal{X}_2/\mathcal{O}_{2}}^{\mathbb{L}}$.
    This proves the surjectivity of the map in the statement.
    For injectivity, suppose that two isomorphisms of schemes in the left-hand side induce the same element of $\Isom^{\mathrm{GP},\Out}$.
    Then the induced isomorphisms of $G_{K_1,S_1}$ and $G_{K_2,S_2}$ differ by an inner automorphism, so the uniqueness assertion in Theorem~\ref{Thm:shimizurestriction} implies that the two base isomorphisms coincide.
    After this common base isomorphism is fixed, the uniqueness assertion in Theorem~\ref{thm:relative_anabelian} implies that the two isomorphisms of curves coincide.
    Hence the map is injective, and therefore bijective.
\end{proof}

\subsection{Reconstruction of the geometric generic fiber}\label{subsection:proof_of_semi-abs-geometric}

\subsubsection{}
In this subsection, we reconstruct the geometric generic fibers of hyperbolic curves over $\mathcal{O}_{K,S}$ from their \'etale fundamental groups.
This is related to the semi-absolute version of the Grothendieck conjecture over $\mathcal{O}_{K,S}$.
We first recall the following two results:

\begin{proposition}\label{local_correspondence}
    Let $i$ range over $\{1,2\}$.
    Let $K_i$ be a number field, $S_i\subset \mathfrak{Primes}_{K_i}$, and let $\mathbb{L}_i$ be a nonempty set of rational primes that are invertible in $\mathcal{O}_{K_i,S_i}$.
    Let $\mathcal{X}_i$ be a smooth curve over $\mathcal{O}_{K_i,S_i}$.
    Let $\Phi:\Pi_{\mathcal{X}_1}^{(\mathbb{L}_1)}\xrightarrow{\sim}\Pi_{\mathcal{X}_2}^{(\mathbb{L}_2)}$ be a Galois-preserving isomorphism.
    Assume the following conditions:
    \begin{enumerate}[(a)]
        \item\label{semi-abs-dec_cond1}
              At least one of $\mathcal{X}_1$ and $\mathcal{X}_2$ is hyperbolic.
        \item\label{semi-abs-dec_cond2}
              $\#\mathbb{L}(S_1)\geq 2$ and $\#\mathbb{L}(S_2)\geq 2$.
    \end{enumerate}
    Then the following hold:
    \begin{enumerate}[(1)]
        \item\label{semi-abs-dec1}
              Let $\Phi_G:G_{K_1,S_1}\xrightarrow{\sim}G_{K_2,S_2}$ be the isomorphism induced by $\Phi$.
              Then $\Phi_G$ induces a local correspondence between $S_1$ and $S_2$ in the sense of \cite[Definition~2.5]{MR4402495},
              i.e., for every prime $\overline{\mathfrak{p}_1}$ of $K_{1,S_1}$ above $S_{1}$, there is a unique prime $\Phi_{G,\ast}(\overline{\mathfrak{p}_1})$ of $K_{2,S_2}$ above $S_{2}$ with $\Phi_G(D_{\overline{\mathfrak{p}_1}}) = D_{\Phi_{G,\ast}(\overline{\mathfrak{p}_1})}$ such that $\Phi_{G,\ast}$ is a bijection between the sets of all primes of $K_{1,S_1}$ above $S_1$ and of $K_{2,S_2}$ above $S_2$.
        \item\label{semi-abs-dec2}
              $\mathbb{L}(S_1)=\mathbb{L}(S_2)$.
        \item\label{semi-abs-dec3}
              For every $p\in\mathbb{L}(S_1)=\mathbb{L}(S_2)$, the isomorphism $\Phi_{G,\ast}$ in \ref{semi-abs-dec1} induces a bijection between the sets of all primes above $p$ of $K_{1,S_1}$ and of $K_{2,S_2}$.
    \end{enumerate}
\end{proposition}

\begin{proof}
    The hypothesis \ref{semi-abs-dec_cond1} and Lemma~\ref{gallem}\ref{gallem1} imply that both $\mathcal{X}_{1}$ and $\mathcal{X}_{2}$ are hyperbolic.
    Moreover, Lemma~\ref{gallem}\ref{gallem2} implies that $\mathbb{L}_{1}=\mathbb{L}_{2}$, and Proposition~\ref{cor:recon_sev}\ref{cor:recon_sev4} implies that $\chi^{p\text{-}\mathrm{cycl}}_{2}\circ \Phi_{G}=\chi^{p\text{-}\mathrm{cycl}}_{1}$ for $p\in \mathbb{L}_{1}(=\mathbb{L}_{2})$.
    By using this data and the hypothesis \ref{semi-abs-dec_cond2}, we obtain the desired bijection $\Phi_{G,\ast}$ induced by $\Phi_G$ (see \cite[Theorem~1.1]{MR3227527}).
    This completes the proof of \ref{semi-abs-dec1}.
    The hypothesis \ref{semi-abs-dec_cond2} and \cite[Proposition~2.8]{MR4402495} imply that \ref{semi-abs-dec2} holds and that the good local correspondence between the sets of primes of $K_{1}$ above $\mathbb{L}(S_{1})$ and of $K_{2}$ above $\mathbb{L}(S_{2})$ holds for $\Phi_G$ (see \cite[Definition~2.5]{MR4402495} for the definition).
    This completes the proof of \ref{semi-abs-dec3}.
\end{proof}

\begin{lemma}\label{lem:deco_is_full}
    Let $K$ be a number field and let $S\subset \mathfrak{Primes}_{K}$.
    Assume that $\#\mathbb{L}(S)\geq 2$.
    Then, for each $\mathfrak{p}\in\mathfrak{Primes}_{K}$ above $\mathbb{L}(S)$, the canonical surjection from $G_{K_\mathfrak{p}}$ to the decomposition subgroup $D_\mathfrak{p}$ of $G_{K,S}$ at $\mathfrak{p}$ is an isomorphism.
\end{lemma}

\begin{proof}
    Let $\mathbb{L}(S)(K)$ be the set of all primes of $K$ above $\mathbb{L}(S)$.
    By considering the composite of canonical morphisms
    \begin{equation*}
        G_{K,S} \twoheadrightarrow G_{K,\mathbb{L}(S)(K)} \hookrightarrow G_{\mathbb{Q},\mathbb{L}(S)},
    \end{equation*}
    we are reduced to the case where $K=\mathbb{Q}$.
    Then the assertion follows from \cite[Th\'eor\`eme 5.1]{MR2476781}.
\end{proof}

\subsubsection{}
From the above results and the Grothendieck conjecture for hyperbolic curves over $p$-adic local fields (see~\cite[Theorem~3.12]{Mochizuki-Tsujimura:RIMS1974}), we obtain the main theorem of this subsection as follows:

\begin{theorem}\label{thm:semi-abs-geometric}
    Let $i$ range over $\{1,2\}$.
    Let $K_i$ be a number field, $S_i\subset \mathfrak{Primes}_{K_i}$, and let $\mathbb{L}_i$ be a nonempty set of rational primes that are invertible in $\mathcal{O}_{K_i,S_i}$.
    Let $\eta_{i}$ be the generic point of $\Spec(\mathcal{O}_{K_{i},S_{i}})$.
    Let $\mathcal{X}_i$ be a smooth curve over $\mathcal{O}_{K_i,S_i}$, and set $X_{i}\coloneqq \mathcal{X}_{i,\eta_{i}}$.
    Let $\Phi:\Pi_{\mathcal{X}_1}^{(\mathbb{L}_1)}\xrightarrow{\sim}\Pi_{\mathcal{X}_2}^{(\mathbb{L}_2)}$ be a Galois-preserving isomorphism and let $\Phi_G:G_{K_1,S_1}\xrightarrow{\sim}G_{K_2,S_2}$ be the isomorphism induced by $\Phi$.
    Assume the following conditions:
    \begin{enumerate}[(a)]
        \item \label{thm:semi-abs-geometric3-a}
              At least one of $\mathcal{X}_1$ and $\mathcal{X}_2$ is hyperbolic.
        \item \label{thm:semi-abs-geometric3-b}
              At least one of $\#\mathbb{L}_1$ and $\#\mathbb{L}_2$ is greater than or equal to $2$.
    \end{enumerate}
    Then the following hold:
    \begin{enumerate}[(1)]
        \item\label{thm:semi-abs-geometric3-1}
              $\Phi_G$ induces a local correspondence between $S_1$ and $S_2$ in the sense of \cite[Definition~2.5]{MR4402495}.
        \item\label{thm:semi-abs-geometric3-2}
              Let $\overline{\mathfrak{p}_1}$ be a prime of $K_{1,S_1}$ above $\mathbb{L}_1$.
              Let $\overline{\mathfrak{p}_{2}}$ be a unique prime of $K_{2,S_2}$ above $S_2$ satisfying $\Phi_G(D_{\overline{\mathfrak{p}_1}}) = D_{\overline{\mathfrak{p}_{2}}}$, whose existence is guaranteed in \ref{thm:semi-abs-geometric3-1}.
              Set $\mathfrak{p}_{i}\coloneqq \overline{\mathfrak{p}}_{i}\mid_{K_{i}}$.
              Let $K^{h}_{i,\mathfrak{p}_i}$ be the fraction field of the henselization of the local ring $(\mathcal O_{K_i})_{\mathfrak p_i}$.
              Then there exists a pair of isomorphisms that fits into the commutative diagram
              \begin{equation*}
                  \vcenter{
                  \xymatrix{
                  X_{1,K_{1,\mathfrak{p}_1}^{h}} \ar[d]\ar[r]^-{\sim}&X_{2,K_{2,\mathfrak{p}_2}^{h}}\ar[d]\\
                  \Spec(K_{1,\mathfrak{p}_1}^{h})\ar[r]^-{\sim} &\Spec(K_{2,\mathfrak{p}_2}^{h}).
                  }
                  }
              \end{equation*}
              In particular, $X_{1,\overline{K_{1}}}$ and $X_{2,\overline{K_{2}}}$ are isomorphic as schemes.
    \end{enumerate}
\end{theorem}

\begin{proof}
    \ref{thm:semi-abs-geometric3-1}
    The hypothesis \ref{thm:semi-abs-geometric3-a} and Lemma~\ref{gallem}\ref{gallem1} imply that both $\mathcal{X}_{1}$ and $\mathcal{X}_{2}$ are hyperbolic.
    Moreover, Lemma~\ref{gallem}\ref{gallem2} implies that $\mathbb{L}_{1}=\mathbb{L}_{2}$.
    Therefore, the assertion follows from the hypothesis \ref{thm:semi-abs-geometric3-b} and Proposition~\ref{local_correspondence}\ref{semi-abs-dec1}.

    \noindent
    \ref{thm:semi-abs-geometric3-2}
    For simplicity, we write $\mathbb{L}$ instead of $\mathbb{L}_{1}(=\mathbb{L}_{2})$.
    Let $p\coloneqq\overline{\mathfrak{p}}_{1}\mid_{\mathbb{Q}} \in\mathbb{L}$.
    By Proposition~\ref{local_correspondence}\ref{semi-abs-dec3}, we obtain that $\overline{\mathfrak{p}}_{2}$ is also above $p$.
    We fix an injection $\iota_{i}:\overline{K_{i}}\hookrightarrow \overline{K_{i,\mathfrak{p}_i}}$.
    Then $D_{\overline{\mathfrak{p}_i}} \subset G_{K_{i},S_i}$ is identified, via the chosen embedding $\iota_i$, with $\mathrm{Gal}(\overline{K_{i,\mathfrak{p}_{i}}}/K_{i,\mathfrak{p}_i})$ by Lemma~\ref{lem:deco_is_full}, and hence the inverse image of $D_{\overline{\mathfrak{p}_i}}$ via $\Pi_{\mathcal{X}_{i}}^{(\mathbb{L})} \twoheadrightarrow G_{K_{i},S_{i}}$ is identified with $\Pi_{X_{i,K_{i,\mathfrak{p}_i}}}^{(\mathbb{L})}$.
    Then $\Phi$ induces a pair $(\Phi_{\overline{\mathfrak{p}_{1}}}, \Phi_{G,\overline{\mathfrak{p}_{1}}})$ of isomorphisms that fits into the commutative diagram
    \begin{equation*}
        \vcenter{
        \xymatrix{
        \Pi_{X_{1,K_{1,\mathfrak{p}_1}}}^{(\mathbb{L})} \ar[d]\ar[r]^-{\Phi_{\overline{\mathfrak{p}_{1}}}}&\Pi_{X_{2,K_{2,\mathfrak{p}_2}}}^{(\mathbb{L})}\ar[d]\\
        G_{K_{1,\mathfrak{p}_1}} \ar[r]^-{\Phi_{G,\overline{\mathfrak{p}_{1}}}} &G_{K_{2,\mathfrak{p}_2}}.
        }
        }
    \end{equation*}
    Therefore, the hypothesis \ref{thm:semi-abs-geometric3-b} and the (semi-)absolute Grothendieck conjecture for hyperbolic curves over $p$-adic local fields (see~\cite[Theorem~3.12]{Mochizuki-Tsujimura:RIMS1974}) imply that there exists a unique pair $(\phi_{\overline{\mathfrak{p}_{1}}}, \phi_{G,\overline{\mathfrak{p}_{1}}})$ of isomorphisms that fits into the commutative diagram
    \begin{equation*}
        \vcenter{
        \xymatrix{
        X_{1,K_{1,\mathfrak{p}_1}} \ar[d]\ar[r]^-{\phi_{\overline{\mathfrak{p}_{1}}}}&X_{2,K_{2,\mathfrak{p}_2}}\ar[d]\\
        \Spec(K_{1,\mathfrak{p}_1})\ar[r]^-{\phi_{G,\overline{\mathfrak{p}_{1}}}} &\Spec(K_{2,\mathfrak{p}_2})
        }
        }
    \end{equation*}
    and whose induced pair of isomorphisms coincides with the pair $(\Phi_{\overline{\mathfrak{p}_{1}}}, \Phi_{G,\overline{\mathfrak{p}_{1}}})$ up to unique inner isomorphism of $\Delta_{X_{2,K_{2,\mathfrak{p}_{2}}}/K_{2,\mathfrak{p}_{2}}}^{\mathbb{L}}$.
    Let
    \begin{equation*}
        \underline{\Isom}_{K_{1,\mathfrak{p}_1}}
        \coloneqq \underline{\Isom}_{K_{1,\mathfrak{p}_1}}\bigl(X_{1,K_{1,\mathfrak{p}_1}},\,X_{2,K_{1,\mathfrak{p}_1}}\bigr)
    \end{equation*}
    be the $K_{1,\mathfrak{p}_1}$-scheme that represents the isomorphism functor of the hyperbolic curves.
    Here, note that the scheme $X_{2,K_{1,\mathfrak{p}_1}}$ is the base change of $\mathcal{X}_{2}$ via the morphism
    \begin{equation*}    \Spec(K_{1,\mathfrak{p}_1})\xrightarrow{\phi_{G,\overline{\mathfrak{p}_{1}}}}\Spec(K_{2,\mathfrak{p}_2})\to \Spec(K_{2})\to\Spec(\mathcal{O}_{K_{2},S_{2}}).
    \end{equation*}
    Then the isom-scheme is finite and unramified over $K_{1,\mathfrak{p}_1}$ by \cite[Theorem 1.11]{MR0262240}, and hence \'etale, since $K_{1,\mathfrak{p}_1}$ is a field.
    In particular, the existence of the pair $(\phi_{\overline{\mathfrak{p}_{1}}}, \phi_{G,\overline{\mathfrak{p}_{1}}})$ of isomorphisms implies that
    \begin{equation*}    \Hom_{K_{1,\mathfrak{p}_1}}\bigl(\Spec(K_{1,\mathfrak{p}_1}),\,\underline{\Isom}_{K_{1,\mathfrak{p}_1}}\bigr)
        \neq \emptyset.
    \end{equation*}
    Let $(\mathcal{O}_{K_{i}})_{\mathfrak{p}_i}^{h}$ be the henselization of the local ring $(\mathcal{O}_{K_{i}})_{\mathfrak{p}_i}$.
    The field isomorphism $\phi_{G,\overline{\mathfrak{p}_{1}}}:K_{2,\mathfrak{p}_2}\xrightarrow{\sim}K_{1,\mathfrak{p}_1}$ fixes $\mathbb{Q}$ and hence sends elements algebraic over $\mathbb{Q}$ to elements algebraic over $\mathbb{Q}$. Under the canonical embeddings of the henselizations into the completions, $K_{i,\mathfrak p_i}^{h}$ is the subfield of $K_{i,\mathfrak p_i}$ consisting of elements algebraic over $K_i$, equivalently algebraic over $\mathbb{Q}$. Therefore $\phi_{G,\overline{\mathfrak{p}_{1}}}$ restricts to an isomorphism $\Spec(K_{1,\mathfrak{p}_1}^{h})\xrightarrow{\sim}\Spec(K_{2,\mathfrak{p}_2}^{h})$.
    In the same way as above, we obtain that the $K_{1,\mathfrak{p}_1}^{h}$-scheme
    \begin{equation*}
        \underline{\Isom}_{K_{1,\mathfrak{p}_1}^{h}}
        \coloneqq   \underline{\Isom}_{K_{1,\mathfrak{p}_1}^{h}}\bigl(X_{1,K_{1,\mathfrak{p}_1}^{h}},X_{2,K_{1,\mathfrak{p}_1}^{h}}\bigr)
    \end{equation*}
    is also finite and \'etale.
    Since the natural morphism $\Spec(K_{1,\mathfrak{p}_1})\to \Spec(K_{1,\mathfrak{p}_1}^{h})$ induces the equivalence of categories
    \begin{equation*}
        \text{F\'Et}(\Spec(K_{1,\mathfrak{p}_1}^{h}))\xrightarrow{\sim}\text{F\'Et}(\Spec(K_{1,\mathfrak{p}_1}))
    \end{equation*}
    by \cite[Chapitre IV Proposition 18.5.15]{MR238860}, we have that
    \begin{equation*}         \Hom_{K_{1,\mathfrak{p}_1}^{h}}\bigl(\Spec(K_{1,\mathfrak{p}_1}^{h}),\,\underline{\Isom}_{K_{1,\mathfrak{p}_1}^{h}}\bigr)
        \xrightarrow{\sim}      \Hom_{K_{1,\mathfrak{p}_1}}\bigl(\Spec(K_{1,\mathfrak{p}_1}),\,\underline{\Isom}_{K_{1,\mathfrak{p}_1}}\bigr)
        \neq \emptyset.
    \end{equation*}
    Hence we obtain a pair of isomorphisms that fits into the commutative diagram
    \begin{equation*}
        \vcenter{
        \xymatrix{
        X_{1,K_{1,\mathfrak{p}_1}^{h}} \ar[d]\ar[r]^-{\sim}&X_{2,K_{2,\mathfrak{p}_2}^{h}}\ar[d]\\
        \Spec(K_{1,\mathfrak{p}_1}^{h})\ar[r]^-{\sim} &\Spec(K_{2,\mathfrak{p}_2}^{h}).
        }
        }
    \end{equation*}
    Since the extension $K_{i,\mathfrak{p}_i}^{h}/K_i$ is algebraic, the geometric generic fibers of $\mathcal{X}_1$ and $\mathcal{X}_2$ are isomorphic as schemes. This completes the proof.
\end{proof}

\section*{Acknowledgements}
The authors would like to express their deepest gratitude to Professor Akio~Tamagawa for suggesting the topic of the present paper and for offering valuable insights through discussions.
They would also like to thank Professor Yuichiro~Taguchi, Professor Yuichiro~Hoshi, and Shun~Ishii for their valuable suggestions and feedback.
Any remaining errors are the authors' own.
This work was supported by JSPS KAKENHI Grant Numbers 25KJ0125 and 23KJ0881.

\printbibliography

\end{document}